\documentclass[]{interact}

\usepackage{epstopdf}
\usepackage[caption=false]{subfig}

\usepackage[natbibapa,nodoi]{apacite}
\theoremstyle{plain}
\newtheorem{theorem}{Theorem}[section]
\newtheorem{lemma}[theorem]{Lemma}

\theoremstyle{definition}
\newtheorem{definition}[theorem]{Definition}

\newtheorem{assumption}[theorem]{Assumption}

\theoremstyle{remark}
\newtheorem{remark}{Remark}

\begin{document}

\articletype{}

\title{Long-Run Average Reward Maximization of A Regulated Regime-Switching Diffusion Model}

\author{
	\name{Lingjia Zeng\textsuperscript{a}\thanks{Corresponding author: Manman Li. } and Manman Li\textsuperscript{b}}
	\affil{College of Mathematics and Statistics, Chongqing University, Chongqing, China;
		\\ \textsuperscript{a}Email:20191891@cqu.edu.cn; \textsuperscript{b}Email: lmm@cqu.edu.cn}
}
\maketitle

\begin{abstract}
This study considers an optimal reinsurance, investment, 
and dividend strategy control problem for insurance companies 
in a regulated Markov regime-switching environment, intending to maximize long-run average reward. Unlike existing single or dual strategy studies, 
an integrated control framework is established under solvency constraints, 
allowing investment and dividends only when the surplus process exceeds a minimum cash requirement level. 
To address the analytical difficulties associated with solving HJB equations and stationary distributions 
in high-dimensional state spaces under regime switching, we construct a numerical approximation scheme for the optimal strategy function 
based on Markov chains and neural networks. Furthermore, we establish the convergence of the corresponding sequence of surplus processes and rigorously prove 
that the associated optimal values converge to the true value function. Finally, we provide a numerical example based on the approximate 
dynamic programming method to demonstrate the feasibility of the proposed method. 
\end{abstract}


\section{Introduction}

  Research on the optimal control problems of insurance companies can be traced back 
  to the classical dividend model introduced by \cite{de1957impostazione}. Since then, 
  surplus management has become a central topic in actuarial science and insurance 
  finance, focusing on how to balance shareholder returns against corporate solvency. 
  Under modern regulatory frameworks such as Solvency II (\cite{EU2009SolvencyII}), 
  insurers are required not only to provide stable returns to shareholders but also 
  to maintain sufficient capital to satisfy stringent solvency requirements. This 
  dual obligation highlights the dynamic tension between profitability and regulatory 
  compliance. Reinsurance, investment, and dividend distribution constitute the 
  three primary decision variables available to insurance companies, and how to 
  optimally coordinate these instruments has remained a critical issue in both 
  academic research and practical applications.

  Within the tradition of discounted criteria, \cite{asmussen1997controlled} showed 
  in a diffusion risk model that the optimal 
  dividend strategy takes the form of a barrier policy, laying the 
  groundwork for subsequent studies. \cite{gerber2006optimal}, in the 
  context of a compound Poisson model, introduced threshold strategies 
  with constrained dividend rates, thereby extending the scope of feasible 
  dividend policies. \cite{avanzi2009strategies} provided a systematic review of the 
  major streams and methods in dividend optimization, summarizing key 
  results and methodological developments under the discounted framework. 
  Collectively, these contributions establish a rich structural understanding 
  of dividend optimization under discounted objectives. However, because this 
  criterion emphasizes short-term discounted returns, the resulting “optimal” 
  strategies often fail to safeguard long-term financial stability or adequately 
  mitigate bankruptcy risk.

  To address this limitation, researchers began incorporating solvency 
  constraints into discounted optimization problems. \cite{he2008optimal} embedded 
  solvency requirements into proportional reinsurance models and demonstrated that 
  regulatory restrictions can substantially alter optimal boundaries and reinsurance 
  intensity. \cite{liang2011optimal} extended the analysis to a joint dividend–investment 
  framework, further revealing the significant impact of capital constraints on optimal 
  strategies. More recently, \cite{avanzi2024optimal} investigated reinsurance 
  design problems with solvency constraints defined by risk measures such as Value-at-Risk
  and Expected Shortfall. Using martingale methods, they derived optimal reinsurance 
  structures combining proportional and stop-loss contracts, thereby highlighting the 
  influence of regulatory risk measures on strategic design. At the institutional level, 
  the Solvency II Directive explicitly incorporates risk-based 
  capital requirements, reinforcing the integrated trade-off between profitability, risk, 
  and regulatory capital. Despite these advances, most studies remain anchored in the 
  discounted paradigm and focus primarily on single or dual control problems, leaving 
  systematic analysis of the joint reinsurance–investment–dividend problem relatively scarce.

  Compared with the discounted framework, the average criterion is more consistent 
  with the long-term nature of insurance liabilities and business operations. 
  As early as 1983, \cite{tapiero1983optimal} proposed an optimal investment–dividend 
  control model under regulatory constraints, adopting long-run average return as the objective. 
  Their study was among the first to integrate solvency requirements with the average 
  reward criterion, illustrating how regulatory boundaries shape optimal corporate 
  strategies. \cite{Li2011} discussed the long-run average return under a heavy-tailed 
  claim model. \cite{sennott1989average} established sufficient conditions for the existence of 
  average-optimal stationary policies in Markov decision processes with infinite state 
  spaces and unbounded costs, advancing the theoretical foundation of average-cost 
  control. In more general settings, \cite{feinberg2012average} provided sufficient 
  conditions for average optimality and the validity of the average cost optimality 
  equation/inequality under weakly continuous transition kernels. On the computational 
  side, \cite{JinYangYin2018} demonstrated the effectiveness of controlled Markov 
  chain approximations in solving long-run dividend optimization problems. Complementary 
  contributions such as \cite{jin2013numerical} and the numerical framework 
  of \cite{kushner2001numerical} offer implementable approaches for high-dimensional stochastic 
  control problems.

  Meanwhile, real-world financial markets exhibit pronounced 
  regime-switching behavior. \cite{sotomayor2011classical} examined 
  dividend optimization under regime switching and showed how state transitions 
  alter the structure and robustness of optimal strategies. \cite{zhu2014dividend} studied optimal 
  constrained dividend rates in a regime-switching diffusion model. 
  \cite{tan2018optimal} integrated debt constraints into a hybrid regime-switching 
  jump–diffusion setting. \cite{JIANG20191} investigated optimal dividend strategies under 
  a regime-switching jump–diffusion model with completely monotone jump densities. 
  In the multi-control direction, \cite{bi2019optimal} analyzed optimal mean–variance 
  investment–reinsurance strategies in a regime-switching market subject to common 
  shocks. \cite{colaneri2021optimal,colaneri2025optimal} developed models under 
  forward preferences to characterize investment and proportional reinsurance decisions 
  in regime-switching markets. \cite{yan2020optimal} introduced Value-at-Risk 
  constraints into investment–reinsurance optimization, revealing the nonlinear 
  interactions between risk measures and control structures. \cite{zhou2024stochastic} further 
  proposed a stochastic hybrid optimal control framework with a recursive cost functional
  in regime-switching environments, broadening the scope of applicable methodologies. 
  In addition, \cite{YinZhangBadowski2003} established connections between regime-switching 
  diffusions and Markov chain limits, laying a mathematical foundation for subsequent 
  modeling and approximation techniques.

  This paper investigates the optimal reinsurance, investment, 
  and dividend strategies for insurance companies under solvency constraints
  to maximize long-run average reward. The surplus process 
  is modulated by a Markov regime-switching. Unlike the framework of 
  \cite{JinYangYin2018}, we assume that the investment process involves 
  risk assets and risk-free assets, and that the insurer 
  can only make investments and dividend payments when the surplus level 
  exceeds the regulatory threshold level. Given the presence of 
  multiple control variables and regime-switching in this system, 
  obtaining an analytical solution via the Hamilton-Jacobi-Bellman (HJB) equation is 
  virtually impossible. Therefore, we employ a Markov chain approximation method (MCAM) to transform the original stochastic control problem into a Markov decision process (MDP) with a discrete state space. The resulting MDP is then solved numerically using a neural network-based approximate dynamic programming approach.

  The rest of the study is organized as follows. Section 2 introduces 
  the model formulation under Markov regime-switching. Section 3 
  establishes the existence and uniqueness of the invariant measure, followed by the derivation of the dynamic programming 
  equation in Section 4. Section 5 presents the Markov chain 
  approximation method and the associated numerical algorithms, 
  while Section 6 discusses the convergence properties of the proposed 
  method. Section 7 presents a numerical example and conducts extended 
  discussions.

\section{Formulation}\label{Formulation}

  We work with a complete filtered probability space $(\Omega,\mathcal{F},\{\mathcal{F}_t\},\mathbb{P})$, 
  where $\{\mathcal{F}_t\}$ is a right continuous, $\mathbb{P}$-complete filtration. 
  Before introducing the risk model, we first build a Markov switching model to simulate 
  market changes. We use a continuous-time irreducible Markov chain $\alpha(t)$ on the finite 
  state space $\mathcal{M}=\{0,1,\cdots,m_0-1\}$ with generator $Q=(q_{ij})\in\mathbb{R}^{m_0\times m_0}$ and 
  transition probabilities satisfy 
  \begin{equation}\label{regime-switching}
    \begin{aligned}
      \mathbb{P}\{\alpha(t+\Delta)=j \vert\alpha(t)=i\}=
        \begin{cases}
          q_{ij}\Delta +o(\Delta),&j\neq i;\\
          1+q_{ii}\Delta+o(\Delta),&j=i,
        \end{cases}
    \end{aligned}
  \end{equation}
  and $q_{ij}\ge 0$ for $i,j\in \mathcal{M}$ and $j\neq i$ and $q_{ii}=-\sum_{j\neq i}q_{ij}<0$ for each 
  $i\in\mathcal{M}$. Generator $Q$ is assumed to be irreducible. 

  We consider an insurance company with insurance risk modeled by a classical risk process 
  \begin{align}
      X(t)=x+ct-C(t),
  \end{align}
  where $x$ is the initial surplus, $C(t)$ is a compound Poisson process approximation with intensity $\lambda$ 
  and claim size distribution $F$ regulated by Markov chain $\alpha(t)$ as follows:
  \begin{align}
      \mathrm{d}C(t)=\mu_C(\alpha(t))\mathrm{d}t-\sigma_C(\alpha(t))\mathrm{d}W_1(t).
  \end{align}
  When $\alpha(t)=i$, We assume that $\mu_C(i)=\lambda_iE(Y)$, 
  $\sigma_C(i)=\sqrt{\lambda_iE(Y^2)}$, where $E(Y)$ is the finite mean claim size. 
  Then, we suppose the premium rate $c$ has a positive security loading factor $\rho>0$ such that 
  $c(\alpha(t))=(1+\rho)\mu_C(\alpha(t))$.

  Under regulations, we let $K>0$ be the regulatory threshold, which is a given minimal cash requirement level. The insurer has to establish 
  a proportional reinsurance, investment, and dividend strategy. 
  The detailed policy is threefold. 
  \begin{enumerate}
      \item The proportional reinsurance strategy. Reinsurance allows the insurance firm to transfer a rate $1-a$ 
            of each claim it pays and to pay a reinsurance premium 
            $\Psi(a(t),\alpha(t))=(1+\beta)(1-a(t))\mu_C(\alpha(t))$ with the security loading factor $\beta>\rho$. 
            
      \item The investment strategy. All surplus above $K$ is converted into investment in proportion, of which the
            proportion of risk investment with  a return rate $r_1(\alpha(t))$ is $s$, and the rest is risk-free 
            investment with a return rate $r_2\le r_1$. Then, we have 
            \begin{equation}
                \begin{aligned}
                    \mathrm{d}S(t)&=[s(t)r_1(\alpha(t))+(1-s(t)-l(t))r_2](X(t)-K)^+\mathrm{d}t\\
                    &+s(t)\sigma_S(\alpha(t))(X(t)-K)^+\mathrm{d}W_2(t).
                \end{aligned}
            \end{equation}
      \item The dividend strategy. Let the cumulative dividend payments $D(\cdot)$ is a non-decreasing and non-negative $\mathcal{F}_t$-adapted process, 
            which represents the accumulated dividend payments paid up to time $t$. 
            The insurance company pays dividends according to the surplus exceeding $K$ at a limited rate $l$. Then, we have 
            \begin{equation}
                \begin{aligned}
                    \mathrm{d}D(t)=l(t)(X(t)-K)^+\mathrm{d}t.
                \end{aligned}
            \end{equation}
  \end{enumerate}
  Thus, the surplus process $X(t)$ of the insurance company is given by
  \begin{equation}\label{surplus process}
      \begin{aligned}
          \mathrm{d}X(t)&=c(\alpha(t))\mathrm{d}t-\Psi(a(t),\alpha(t))\mathrm{d}t-a(t)\mathrm{d}C(t)+\mathrm{d}S(t)-\mathrm{d}D(t),\\
          (X(0),\alpha(0))&=(x,i),
      \end{aligned}
  \end{equation}
  where:
  \begin{enumerate}
  \item $W_1(t)$ and $W_2(t)$ are independent standard Brownian motions.
  \item $\sigma_C(\cdot),\sigma_S(\cdot)>0$ and $\mu_C(\cdot)>0$ are bounded, i.e. $\sigma_C(\cdot), \sigma_S(\cdot)\le \sigma^2_{\max}$, $\mu_C(\cdot)\le \mu_{\max}$.
  \item We assume $a(t)$, $s(t)$ and $l(t)$ are $\mathcal{F}_t$-adapted processes, and an admissible strategy $u=(a,s,l)$ is 
        progressively measurable with respect to $\{W_1(s),W_2(s),\alpha(s),0\le s\le t\}$.  For ease of computation, the admissible strategy set $\mathcal{U}$ is defined as 
        \begin{align}\label{admissible strategy set}
          \mathcal{U}=\{u\in \mathbb{R}^3: a\in[M_a,1],s\in[0,M_s],l\in[M_l,1], s+l\le1\}.
        \end{align}
  \end{enumerate} 
Figure \ref{figure: sample path} illustrates the dynamics of $X(t)$ and $D(t)$ under a given control strategy.
\begin{figure}[h!]
    \centering
    \captionsetup{font={small, stretch=1.312}}
    \includegraphics[width=0.6\linewidth]{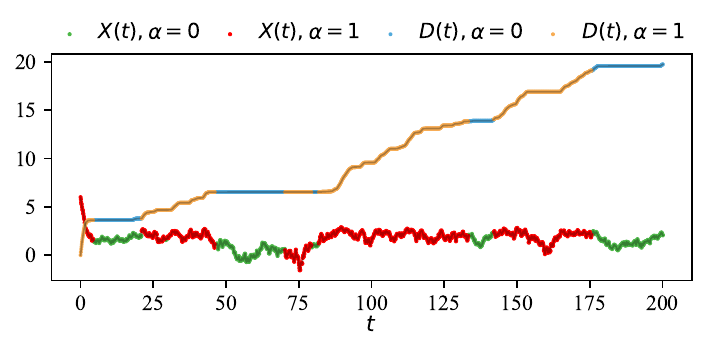}
    \caption{The surplus and dividend processes under a given control strategy.}
    \label{figure: sample path}
\end{figure}

  Then a Borel measurable function $u(x,i)=(a(x,i),s(x,i),l(x,i))$ is an admissible feedback strategy if \eqref{surplus process} has a unique 
  solution. For an arbitrary admissible feedback control $u(\cdot,\cdot)$, we assume that the insurer's objective function is 
  the average reward in the long term given by
  \begin{equation}\label{objective function}
      \begin{aligned}
          J(x,i,u)&=\liminf_{T\to\infty}\frac{1}{T}\mathbb{E}_{x,i}\left[\int^T_0c(\alpha(t))\mathrm{d}t
          +\int^T_0[s(X(t),\alpha(t))r_1(\alpha(t))\right.\\ &+(1-s(X(t),\alpha(t))
          -l(X(t),\alpha(t)))r_2](X(t)-K)^{+}\mathrm{d}t\\
          &\left.-\int^T_0\Psi(a(X(t),\alpha(t)),\alpha(t))\mathrm{d}t-\int^T_0a(t)\mu_C(\alpha(t))\mathrm{d}t\right],
      \end{aligned}
  \end{equation} 
  where $\mathbb{E}_{x,i}$ is the conditional expectation when $(X(0),\alpha(0))=(x,i)$. 
  Let $\mathbb{P}$ denote the conditional probability on $(X(0),\alpha(0))=(x,i)$. 
  If the controlled process $(X(t),\alpha(t))$ admits a unique invariant distribution, 
  then the long-run average reward does not depend on the initial state. 
  Hence, we may denote $J(x,i,u)$ simply by $\gamma(u)$. Define the optimal value as 
  \begin{align}\label{optimal_value}
      \bar{\gamma}:=\sup_{u\in\mathcal{U}}\gamma(u).
  \end{align}

  For the convenience of discussion, we record the drift term 
  of $X(t)$ as $b(X(t),\alpha(t))$ and the diffusion term as $\sigma(X(t),\alpha(t))$, that is
  \begin{equation}
      \begin{aligned}
          b(X(t),\alpha(t))&=c(\alpha(t))-\Psi(a(t),\alpha(t))-a(t)\mu_C(\alpha(t))
          +s(t)r_1(\alpha(t))(X(t)-K)^+\\
          &+(1-s(t)-l(t))r_2(X(t)-K)^+-l(t)(X(t)-K)^+,\\
          \sigma^2(X(t),\alpha(t))&=a^2(t)\sigma^2_C(\alpha(t))+[s(t)\sigma_S(\alpha(t))(X(t)-K)^+]^2.
      \end{aligned}
  \end{equation}
  Let  the drift $b(x,i,u)$ and diffusion coefficient $\sigma(x,i,u)$ satisfy the following conditions:
  \begin{enumerate}
      \item Linear growth: There exists a constant $C>0$ such that for all $x\in \mathbb{R}$, 
      $i \in \mathcal{M}$ and $u\in\mathcal{U}$, 
      \begin{align}
          \vert b(x,i,u)\vert +\vert\sigma(x,i,u)\vert \le C(1+\vert x\vert).
      \end{align}
      \item Local Lipschitz continuity: For every integer $N\ge 1$, there exists a positive 
      constant $M_N$ such that for all $t\in[0,T]$, $i\in\mathcal{M}$, $u\in\mathcal{U}$ and $x,y\in\mathbb{R}$ with
      $\vert x\vert \vee \vert y \vert \le M_N$, we have
      \begin{align}
          \vert b(x,i,u)-b(y,i,u)\vert \vee \vert \sigma(x,i,u)-\sigma(y,i,u)\vert \le M_N\vert x-y\vert.
      \end{align}
  \end{enumerate}
  Under the linear growth and local Lipschitz continuity assumptions on the coefficients 
  $b(\cdot,\cdot,\cdot)$ and $\sigma(\cdot,\cdot,\cdot)$, and given that $\alpha(t)$ is a c\'adl\'ag 
  finite-state Markov chain, it follows from the standard theory of regime-switching SDEs 
  in \cite{yin2010hybrid} 
  that the system admits a unique strong solution $(X(t),\alpha(t))$ to  
  \eqref{surplus process} for any initial data $(X(0),\alpha(0))=(x,i)$.

  For $i\in\mathcal{M}$, an arbitrary strategy $u\in\mathcal{U}$ and 
  a test function $V(\cdot,i)\in C^2(\mathbb{R})$, 
  we define the associated generator $\mathcal{L}^u$  acting on $V$ as follows: 
  \begin{align}
      \mathcal{L}^uV(x,i)=V_x(x,i)b(x,i,u)+\frac{1}{2}\sigma^2(x,i,u)V_{xx}(x,i)+QV(x,\cdot)(i),
  \end{align}
  where $V_x$ and $V_{xx}$ present the first and second order derivatives w.r.t. $x$, and 
  \begin{align}
      QV(x,\cdot)(i)=\sum_{j\neq i}q_{ij}(V(x,j)-V(x,i)).
  \end{align}
  If $\bar{\gamma}$ exists, according to the dynamic programming principle in \cite{fleming2006controlled}, 
  there exists a  viscosity solution $V$ to the following system of coupled HJB equations:
  \begin{align}\label{HJB}
      \bar{\gamma}=\max_{u\in\mathcal{U}}\{\mathcal{L}^uV(x,i)+f(x,i,u)\}, \forall i\in\mathcal{M}.
  \end{align}
  Here  a Borel measurable function $f(\cdot,\cdot,\cdot):\mathbb{R}\times \mathcal{M}\times\mathcal{U}\mapsto \mathbb{R}$ is defined as
    \begin{equation}\label{reward functionf(x)}
   \begin{aligned}
   f(x,i,u)&=c(i)+[sr_1(i)+(1-s-l)r_2](x-K)^+-\Psi(a,i)-a\mu_C(i).
   \end{aligned}
   \end{equation}
  Under suitable regularity assumptions, $V$ can be assumed to be sufficiently smooth 
  such that the operator $\mathcal{L}^u$ is well defined in the classical sense.
  
  In reality, the surplus of an insurance company can not be infinite. 
  For the convenience of calculation, we can select a large enough 
  interval $G=[-B,B]\subset \mathbb{R}$ as the state space of the surplus process $X(t)$.

\section{Invariant measure}
  
  In order to make the long-run average reward exist, for the process $Y(t)=(X(t),\alpha(t))$, 
  we consider a method to replace the instantaneous measures with invariant measures. 
  
  First, we need the following Assumption \ref{assumptionA}.   
  \begin{assumption}
      \label{assumptionA}
      $Y(t)$ is positive recurrent with respect to some bounded domain $E\times \{i\}$, 
      where $E\subset G\subset \mathbb{R}$, $i$ is fixed and $i\in \mathcal{M}$.
  \end{assumption}
  Assumption \ref{assumptionA} can be ensured in our framework, as suitable parameter specifications 
  guarantee that the process $Y(t)$ remains positive recurrent on the prescribed domain. 
  Based on Assumption \ref{assumptionA}, we introduce Lemma \ref{lemma1}.
  \begin{lemma}\label{lemma1}
      Assume Assumption \ref{assumptionA} holds. $Y(t)$ is the positive recurrent with 
      respect to $G\times \mathcal{M}$. 
  \end{lemma}
  \begin{proof}
      The lemma is obviously true by applying Theorem 3.12 in \cite{yin2010hybrid}.
  \end{proof}

  We consider defining a sequence of stopping times $\{\iota_k\}, k=0,1,2,\cdots$. Let $\iota_0=0$, 
  $\iota_{2k+1}$ be the first time after $\iota_{2k}$ when $Y(t)$ reaches the boundary $\partial E\times \{i\}$, 
  and $\iota_{2k+2}$ be the first time after $\iota_{2k+1}$ when $Y(t)$ reaches the boundary 
  $\partial G\times \{i\}$. Thus, the sample path of $Y(t)$ can be divided into the cycles as 
  \begin{align}\label{cycles}
      [\iota_0,\iota_2),[\iota_2,\iota_4),\cdots,[\iota_{2k},\iota_{2k+2}),\cdots.
  \end{align}
  By Lemma \ref{lemma1}, $Y(t)$ is the positive recurrent with respect to $G\times \mathcal{M}$. 
  Thus, the stopping times $\{\iota_k\},k=0,1,2,\cdots$ are finite almost surely. 
  Because the process $Y(t)$ is positive recurrent, we assume the set without loss of generality 
  that $Y(0)=(X(0),\alpha(0))=(x,i)\in \partial G\times \{i\}$. 
  It follows from the strong Markov property of the process $Y(t)$ that the 
  sequence $Y_n=(X_n,i)=Y(\iota_{2n}),n=0,1,\cdots$ is 
  a Markov chain on $\partial G\times \{i\}$. Denote by $\mathcal{B}(\partial G)$ the 
  collection of Borel measurable sets on $\partial G$.  Note that the process $Y(t)$, 
  starting from $(x,i)$, may jump many times before it reaches the 
  set $(H,i)$ where $H\in \mathcal{B}(\partial G)$. 
  The one-step transition probabilities of this Markov chain $Y_n$ are defined as
  \begin{align}
      \tilde{p}(x,H)=\mathbb{P}(Y_1\in(H\times \{i\})\vert Y_0=(x,i)).
  \end{align}
  We can also define the $n$-step transition probability of the Markov chain as $\tilde{p}^{(n)}(x,H)$ 
  for any $n\ge 1$.

  Now we need to construct the stationary distribution of the process $Y(t)=(X(t),\alpha(t))$.

  \begin{theorem}
      \label{stationary}
      The positive recurrent process $Y(t)=(X(t),\alpha(t))$ has a unique stationary distribution 
      $\nu(\cdot,\cdot)=(\nu(\cdot,i):i\in\mathcal{M})$. Denote by $\zeta (\cdot,\cdot)$ the stationary density 
      associated with the stationary distribution $\nu(\cdot,\cdot)$. Then for any $(x,i)\in G\times \mathcal{M}$ and 
      a Borel measurable 
      function $f(\cdot,\cdot,\cdot):\mathbb{R}\times \mathcal{M}\times \mathcal{U}\mapsto \mathbb{R}$ defined in (\ref{reward functionf(x)}) such that
      \begin{equation}\label{reward function condition}
        \begin{aligned}
          \sum^{m_0-1}_{i=0}\int_{\mathbb{R}}\vert f(x,i,u)\vert \zeta(x,i)\mathrm{d}x<\infty,
        \end{aligned}
    \end{equation}
      we have 
      \begin{equation}\label{pb}
        \begin{aligned}
          \mathbb{P}_{x,i}&\left\{\lim_{T\to\infty}\frac{1}{T}\int^T_0 
          [c(\alpha(t))+[s(X(t),\alpha(t))r_1(\alpha(t))+\right.\\
          &(1-s(X(t),\alpha(t))-l(X(t),\alpha(t)))r_2](X(t)-K)^+\\
          &\left.-\Psi(a(X(t),\alpha(t)),\alpha(t))-a(X(t),\alpha(t))\mu_C(\alpha(t))]dt=\bar{\gamma}\right\}=1,
        \end{aligned}
      \end{equation}
      where
      \begin{equation}
        \begin{aligned}
          \bar{\gamma}&=\sum^{m_0-1}_{i=0}\int_{\mathbb{R}}f(x,i,u)\zeta(x,i)\mathrm{d}x\\
         &=\sum^{m_0-1}_{i=0}\int_{\mathbb{R}}\left[c(i)+[sr_1(i)+(1-s-l)r_2](x-K)^+
         -\Psi(a,i)-a\mu_C(i)\right]\zeta(x,i)\mathrm{d}x.
        \end{aligned}
      \end{equation}
  \end{theorem}
  
  \begin{proof}
        
        Firstly, we need to show that \eqref{reward functionf(x)} satisfies the condition of \eqref{reward function condition}. 
        
        Because $\mu_C(i)\le \mu_{\max}$, $\vert c(i)\vert \le(1+\rho)\mu_{\max}$, for $\sum^{m_0-1}_{i=0}\int_{\mathbb{R}}\vert c(i)\vert \zeta(x,i)\mathrm{d}x$, 
        we have 
        \begin{align}\label{cint}
            \sum^{m_0}_{i=1}\int_{\mathbb{R}}\vert c(i)\vert \zeta(x,i)\mathrm{d}x\le (1+\rho)\mu_{\max}\sum^{m_0}_{i=1}\int_{\mathbb{R}}\zeta(x,i)\mathrm{d}x.
        \end{align}
        Since $\zeta(\cdot,\cdot)$ is the stationary density that satisfies $\sum^{m_0}_{i=1}\int_{\mathbb{R}}\zeta(x,i)\mathrm{d}x=1$, the 
        integral in \eqref{cint} is finite. 
        Similarly, 
        \begin{equation}
            \begin{aligned}
                \sum^{m_0}_{i=1}\int_{\mathbb{R}}\vert\Psi(a,i)\vert\zeta(x,i)\mathrm{d}x<\infty,\ \ 
                \sum^{m_0}_{i=1}\int_{\mathbb{R}}\vert a\mu_C(i)\vert\zeta(x,i)\mathrm{d}x<\infty.
            \end{aligned}
        \end{equation}
        
        Next, we define the time of first return to $E=\{x\in\mathbb{R}:\vert x\vert<N_0\}$ as 
        $\tau_{E}=\inf\{t>0:Y(t)=(X(t),\alpha(t))\in E\times \mathcal{M}\}$ and the $n$th return time 
        $\tau_E^{(n)}$ satisfying that $\tau_E^{(0)}=0$ and $\tau_E^{(n)}=\tau_E^{(n-1)}+(\tau_E^{(n)}-\tau_E^{(n-1)})$, 
        where $\Delta \tau_{E,n}=\tau_E^{(n)}-\tau_E^{(n-1)}$ is the interval time from the $(n-1)$th return to the 
        $n$th return. Because $Y(t)$ is positive recurrent, $\tau_E^{(n)}\to\infty$ and
        for any $n$, we have 
        \begin{align}\label{tauEconditon}
            E[\Delta \tau_{E,n}\vert Y(\tau_E^{(n-1)})\in E\times \mathcal{M}]<\infty.
        \end{align}
        Apply Dynkin's formula to the Lyapunov function $V(x,i)=x^2$ 
        for $t\in [\tau_E^{(n)},\tau_E^{(n+1)})$, we get 
        \begin{equation}\label{tauEV}
            \begin{aligned}
                \mathbb{E}[V(X(t),\alpha(t))\vert \mathcal{F}_{\tau_E^{(n)}}]=V(X(\tau_E^{(n)}),\alpha(\tau_E^{(n)}))
                +\mathbb{E}\left[\int^t_{\tau_E^{(n)}}\mathcal{L}^uV(X(s),\alpha(s))\mathrm{d}s\vert \mathcal{F}_{\tau_E^{(n)}}\right].
            \end{aligned}
        \end{equation}
        Since $Y(\tau_E^{(n)})\in E\times \mathcal{M}$, we 
        obtain $V(X(\tau_E^{(n)}),\alpha(\tau_E^{(n)}))=[X(\tau_E^{(n)})]^2\le N_0^2$.
        When $X(s)\in E^c$, we have $\mathcal{L}^uV(X(s),\alpha(s))\le-C$ where $C>0$.
        Thus, 
        \begin{equation}\label{LVneq}
            \begin{aligned}
                \int^t_{\tau_E^{(n)}}\mathcal{L}^uV(X(s),\alpha(s))\mathrm{d}s
                &\le\int^t_{\tau_E^{(n)}}\boldsymbol{1}_{\{X(s)\in E^c\}}\cdot(-C)\mathrm{d}s\\
                &+\int^t_{\tau_E^{(n)}}\boldsymbol{1}_{\{X(s)\in E\}}\cdot\mathcal{L}^uV(X(s),\alpha(s))\mathrm{d}s.
            \end{aligned}
        \end{equation}
        Because $\vert \mathcal{L}V(X(s),\alpha(s))\vert \le N_1$ 
        when $(X(s),\alpha(s))\in E\times\mathcal{M}$, where $N_1$ is a constant, the integral of the second term on 
        the right side of \eqref{LVneq} is finite. By taking conditional 
        expectation $\mathbb{E}[\cdot\vert\mathcal{F}_{\tau_E^{(n)}}]$ 
        for \eqref{LVneq} and scaling it down, 
        we can know that 
        \begin{equation}\label{tauEC}
            \begin{aligned}
                \mathbb{E}\left[\int^t_{\tau_E^{(n)}}\mathcal{L}^uV(X(s),\alpha(s))\mathrm{d}s\vert \mathcal{F}_{\tau_E^{(n)}}\right]
                &\le(N_1-C)\cdot\mathbb{E}[t-\tau_E^{(n)}\vert \mathcal{F}_{\tau_E^{(n)}}]\\
            \end{aligned}
        \end{equation}
        For any $t$, there exists $n$ such that $t\in[\tau_E^{(n)},\tau_E^{(n+1)})$. Combining
        \eqref{tauEV} and \eqref{tauEC} to obtain conditional expectation for any initial 
        state $(x,i)$, we have
        \begin{equation}
            \begin{aligned}
            &\mathbb{E}[V(X(t),\alpha(t))\vert X(0)=x,\alpha(0)=i]
            \le N_0^2
            +(N_1-C)\cdot\mathbb{E}[t-\tau_E^{(n)}\vert X(0)=x,\alpha(0)=i].
            \end{aligned}
        \end{equation}
        Because $\mathbb{E}[\tau_E^{(n+1)}-\tau_E^{(n)}\vert \mathcal{F}_{\tau_E^{(n)}}]<\infty$, we get 
        \begin{align}
            \sup_{t\ge0}\mathbb{E}[V(X(t),\alpha(t))\vert X(0)=x,\alpha(0)=i]<\infty.
        \end{align}
        Thus, 
        \begin{align}
            \mathbb{E}_\nu[X^2]=\lim_{t\to\infty}\mathbb{E}[X^2(t)\vert (X(0),\alpha(0))\sim \nu]<\infty.
        \end{align}
        Through the Cauchy - Schwarz inequality, we can get $\mathbb{E}_{\nu}\vert X\vert \le \sqrt{\mathbb{E}_{\nu}[X^2]}$. 
        Easily, we know that $(x-K)^+\le\vert x\vert +K$, so
        \begin{equation}
            \begin{aligned}
                \sum^{m_0-1}_{i=0}\int_{\mathbb{R}}\vert [sr_1(i)+(1-s-l)r_2](x-K)^+\vert\zeta(x,i)\mathrm{d}x
                &\le M_2\sum^{m_0-1}_{i=0}\int_{\mathbb{R}}(\vert x\vert +K)\zeta(x,i)\mathrm{d}x\\
                &=M_2(\mathbb{E}_{\nu}\vert X\vert +K)<\infty,
            \end{aligned}
        \end{equation}
        where $M_2=\max\{\vert sr_1+(1-s-l)r_2\vert\}$ is a bounded constant. 
        From the above, we know that \eqref{reward functionf(x)} satisfies the condition of \eqref{reward function condition}. 
        
        According to Lemma 4.1 in \cite{yin2010hybrid}, $Y_n(t)=(X_n,i)$ has a unique stationary distribution $\varphi(\cdot)$. 
        Recall the cycles defined in \eqref{cycles}. 
        For any $H\in \mathcal{B}(\mathbb{R})$,  
        $\varphi(H)=\lim_{n\to\infty}\tilde{p}^{(n)}(x,H)$. Denote by $\tau^{H\times\{i\}}$ the time 
        spent by the path of $Y(t)$ in the set $(H\times \{i\})$ during the first cycle. Set 
        \begin{align}\label{measure1}
            \hat{\nu}(H,i):=\int_{\partial G}\varphi(\mathrm{d}x)\mathbb{E}_x\tau^{H\times\{i\}}.
        \end{align}
        Using the proof of Theorem 4.3 in \cite{yin2010hybrid}, we obtain
        \begin{equation}\label{measure2}
            \begin{aligned}
                &\sum^{m_0-1}_{i=0}\int_\mathbb{R}[c(i)+[sr_1(i)+(1-s-l)r_2](x-K)^+-\Psi(a,i)-a\mu_C(i)]\hat{\nu}(\mathrm{d}x,i)\\
                &=\int_{\partial G}\varphi(\mathrm{d}x)\mathbb{E}_x\int^{\iota_2}_0
                [c(\alpha(t))+[s(X(t),\alpha(t))r_1(\alpha(t))+(1-s(X(t),\alpha(t))\\
                &-l(X(t),\alpha(t)))r_2](X(t)-K)^+-\Psi(a(X(t),\alpha(t)),\alpha(t))\\
                &-a(X(t),\alpha(t))\mu_C(\alpha(t))]\mathrm{d}t\\
                &=\sum^{m_0-1}_{i=0}\int_{\mathbb{R}}\mathbb{E}_{x,i}[c(\alpha(t))+[s(X(t),\alpha(t))r_1(\alpha(t))+(1-s(X(t),\alpha(t))\\
                &-l(X(t),\alpha(t)))r_2](X(t)-K)^+-\Psi(a(X(t),\alpha(t)),\alpha(t))\\
                &-a(X(t),\alpha(t))\mu_C(\alpha(t))]\hat{\nu}(\mathrm{d}x,i).
            \end{aligned}
        \end{equation}
        Thus, the desired stationary distribution is defined by the normalized measure as 
        \begin{align}\label{stationarymeasure}
            \nu(H,i)=\frac{\hat{\nu}(H,i)}{\sum^{m_0-1}_{j=0}\hat{\nu}(\mathbb{R},j)},\forall i\in\mathcal{M}.
        \end{align}
        
        Next, we will verify \eqref{pb}. We know about the stationary distribution that any 
        initial distribution starting from any point $(x, i)$ is asymptotically equivalent to the stationary 
        distribution starting from the initial distribution. Then we only need to verify that the initial 
        distribution is the latter case. 
        
        If the initial distribution is the stationary distribution 
        of the Markov chain $Y_n$, for any $H\in\mathcal{B}(\partial G)$, we have 
        \begin{align}
            \mathbb{P}\{(X(0),\alpha(0))\in(H\times \{i\})\}=\varphi(H).
        \end{align}
        Set a sequence of random variables 
        \begin{equation}
        \begin{aligned}
            \eta_n&=\int_{\iota_{2n}}^{\iota_{2n+2}}[c(\alpha(t))+[s(X(t),\alpha(t))r_1(\alpha(t))+(1-s(X(t),\alpha(t))\\
            &-l(X(t),\alpha(t)))r_2](X(t)-K)^+-\Psi(a(X(t),\alpha(t)),\alpha(t))\\
            &-a(X(t),\alpha(t))\mu_C(\alpha(t))]\mathrm{d}t.
        \end{aligned}
        \end{equation}
        From \eqref{measure1} and \eqref{measure2}, we obtain 
        \begin{equation}\label{expectation}
            \begin{aligned}
            \mathbb{E}\eta_n=\sum^{m_0-1}_{i=0}\int_{\mathbb{R}}[c(i)+[sr_1(i)+(1-s-l)r_2](x-K)^+
            -\Psi(a,i)-a\mu_C(i)]\hat{\nu}(\mathrm{d}x,i),
            \end{aligned}
        \end{equation}
        for all $n=0,1,2,\cdots$. Let $\phi(T)$ denote the number of cycles completed up to 
        time $T$. That is, 
        \begin{align}
            \phi(T):=\max\left\{n\in\mathbb{N}:\sum^n_{k=1}(\iota_{2k}-\iota_{2k-2})\le T\right\}.
        \end{align} 
        Thus, $\int^T_0f(X(t),\alpha(t))\mathrm{d}t$ can be decomposed into 
        \begin{equation}\label{func}
            \begin{aligned}
            &\int^T_0[c(\alpha(t))+[s(X(t),\alpha(t))r_1(\alpha(t))+(1-s(X(t),\alpha(t))\\
            &-l(X(t),\alpha(t)))r_2](X(t)-K)^+-\Psi(a(X(t),\alpha(t)),\alpha(t))\\
            &-a(X(t),\alpha(t))\mu_C(\alpha(t))]\mathrm{d}t\\
            &=\sum^{\phi(T)}_{n=0}\eta_n+\int^T_{\iota_{2\phi(T)}}[c(\alpha(t))+[s(X(t),\alpha(t))r_1(\alpha(t))+(1-s(X(t),\alpha(t))\\
            &-l(X(t),\alpha(t)))r_2](X(t)-K)^+
            -\Psi(a(X(t),\alpha(t)),\alpha(t))\\
            &-a(X(t),\alpha(t))\mu_C(\alpha(t))]\mathrm{d}t.
            \end{aligned}
        \end{equation}
        Without loss of generality, we can assume $f(x,i)=f_1(x,i)-f_2(x,i)$, 
        where $f_1(x,i)\ge 0$, $f_2(x,i)\ge 0$. 
        Then we have 
        \begin{align}
            \eta_{n1}=\int_{\iota_{2n}}^{\iota_{2n+2}}f_1(X(t),\alpha(t))\mathrm{d}t, \eta_{n2}=\int^{\iota_{2n+2}}_{\iota_{2n}}f_2(X(t),\alpha(t))\mathrm{d}t.
        \end{align}
        Similarly, we have
        \begin{equation}
            \begin{aligned}
                \int^T_0f_1(X(t),\alpha(t))\mathrm{d}t&=\sum^{\phi(T)}_{n=0}\eta_{n1}
                +\int^T_{\iota_{2\phi(T)}}f_1(X(t),\alpha(t))\mathrm{d}t,\\ 
                \int^T_0f_2(X(t),\alpha(t))\mathrm{d}t&=\sum^{\phi(T)}_{n=0}\eta_{n2}
                +\int^T_{\iota_{2\phi(T)}}f_2(X(t),\alpha(t))\mathrm{d}t.
            \end{aligned}
        \end{equation}
        By $f_1(x,i), f_2(x,i)\ge 0$, we can easily get 
        \begin{equation}
            \begin{aligned}
                \sum^{\phi(T)}_{n=0}\eta_{n1}\le \int^T_0 f_1(X(t),\alpha(t))\mathrm{d}t\le \sum^{\phi(T)+1}_{n=0}\eta_{n1},
                \sum^{\phi(T)}_{n=0}\eta_{n2}\le \int^T_0 f_2(X(t),\alpha(t))\mathrm{d}t\le \sum^{\phi(T)+1}_{n=0}\eta_{n2}.
            \end{aligned}
        \end{equation}
        Thus, we obtain 
        \begin{equation}
            \begin{aligned}
            \sum^{\phi(T)}_{n=0}\eta_n&\le \int^T_0[c(\alpha(t))+[s(X(t),\alpha(t))r_1(\alpha(t))+(1-s(X(t),\alpha(t))\\
            &-l(X(t),\alpha(t)))r_2](X(t)-K)^+-\Psi(a(X(t),\alpha(t)),\alpha(t))\\
            &-a(X(t),\alpha(t))\mu_C(\alpha(t))]\mathrm{d}t\\
            &\le \sum^{\phi(T)+1}_{n=0}\eta_n.
            \end{aligned}
        \end{equation}
        Then
        \begin{equation}
            \begin{aligned}
            \frac{1}{\phi(T)}\sum^{\phi(T)}_{n=0}\eta_n&\le \frac{1}{\phi(T)}\int^T_0 [c(\alpha(t))+[s(X(t),\alpha(t))r_1(\alpha(t))+(1-s(X(t),\alpha(t))\\
            &-l(X(t),\alpha(t)))r_2](X(t)-K)
            -\Psi(a(X(t),\alpha(t)),\alpha(t))\\
            &-a(X(t),\alpha(t))\mu_C(\alpha(t))]\mathrm{d}t\\
            &\le \frac{1}{\phi(T)}\sum^{\phi(T)+1}_{n=0}\eta_n.
            \end{aligned}
        \end{equation}
        As $T\to \infty$, $\phi(T)\to \infty$. Combining with \eqref{expectation}, we have 
        \begin{equation}
            \begin{aligned}
            &\frac{1}{\phi(T)}\int^T_0 [c(\alpha(t))+[s(X(t),\alpha(t))r_1(\alpha(t))+(1-s(X(t),\alpha(t))\\
            &-l(X(t),\alpha(t)))r_2](X(t)-K)^+
            -\Psi(a(X(t),\alpha(t)),\alpha(t))\\
            &-a(X(t),\alpha(t))\mu_C(\alpha(t))]\mathrm{d}t\\
            &\to \sum^{m_0-1}_{i=0}\int_{\mathbb{R}}[c(i)+[sr_1(i)+(1-s-l)r_2](x-K)^+-\Psi(a,i)-a\mu_C(i)]\hat{\nu}(\mathrm{d}x,i).
            \end{aligned}
        \end{equation}
        In addition, the law of large numbers implies that 
        \begin{equation}\label{approx1}
            \mathbb{P}\left\{
            \begin{aligned}
            &\frac{1}{n}\sum^n_{k=0}\eta_k\to\sum^{m_0}_{i=1}\int_{\mathbb{R}}[c(i)+(sr_1(i)+(1-s-l)r_2)(x-K)^+\\
            &-\Psi(a,i)-a\mu_C(i)]\hat{\nu}(\mathrm{d}x,i), \text{ as } n\to \infty
            \end{aligned}
            \right\}=1.
        \end{equation}
        In particular, when $f(x,i)\equiv 1$, \eqref{approx1} reduces to 
        \begin{align}\label{f(x)1}
            \mathbb{P}\left\{\frac{\iota_{2n+2}}{n}\to \sum^{m_0-1}_{i=0}\hat{\nu}(\mathrm{d}x,i), \text{ as } n\to \infty\right\}=1.
        \end{align}
        Note that $\phi(T)/(\phi(T)+1)\to 1$ almost surely as $T\to \infty$. Thus, it follows from 
        \eqref{f(x)1} that as $T\to\infty$, 
        \begin{align}\label{neq1}
            \frac{\iota_{2\phi(T)}}{\iota_{2\phi(T)+2}}=\frac{\frac{\iota_{2\phi(T)}}{\phi(T)}}{\frac{\iota_{2\phi(T)+2}}{\phi(T)+1}}\frac{\phi(T)}{\phi(T)+1}\to 1 \text{ a.s.}
        \end{align} 
        Since $\iota_{2\phi(T)}\le T\le \iota_{2\phi(T)+2}$, we get 
        \begin{align}
            \frac{\iota_{2\phi(T)}}{\iota_{2\phi(T)+2}}\le \frac{\iota_{2\phi(T)}}{T}\le \frac{\iota_{2\phi(T)}}{\iota_{2\phi(T)}}=1.
        \end{align}
        Thus, we have from \eqref{neq1} that 
        \begin{align}\label{2T1}
            \frac{\iota_{2\phi(T)}}{T}\to 1 \text{ a.s. as }T\to\infty.
        \end{align}
        Therefore, \eqref{f(x)1} implies that 
        \begin{align}\label{phiT1}
            \mathbb{P}\left\{\frac{T}{\phi(T)}\to\sum^{m_0-1}_{i=0}\hat{\nu}(\mathrm{d}x,i), \text{ as } T\to\infty\right\}=1.
        \end{align}
        Next, using \eqref{approx1}, \eqref{2T1} and \eqref{phiT1}, we obtain
        \begin{equation}
            \begin{aligned}
                &\frac{1}{T}\int^T_0[c(\alpha(t))+[s(X(t),\alpha(t))r_1(\alpha(t))+(1-s(X(t),\alpha(t))\\
                &\ \ -l(X(t),\alpha(t)))r_2](X(t)-K)^+-\Psi(a(X(t),\alpha(t)),\alpha(t))
                -a(X(t),\alpha(t))\mu_C(\alpha(t))]\mathrm{d}t\\
                &=\frac{1}{\phi(T)}\times \int^T_0[c(\alpha(t))+[s(X(t),\alpha(t))r_1(\alpha(t))+(1-s(X(t),\alpha(t))\\
                &\ \ -l(X(t),\alpha(t)))r_2](X(t)-K)^+-\Psi(a(X(t),\alpha(t)),\alpha(t))
                -a(X(t),\alpha(t))\mu_C(\alpha(t))]\mathrm{d}t \times \frac{\phi(T)}{T}\\
                &\to\sum^{m_0-1}_{i=0}\int_{\mathbb{R}}[c(i)+[sr_1(i)+(1-s-l)r_2](x-K)^+-\Psi(a,i)
               -a\mu_C(i)]\hat{\nu}(\mathrm{d}x,i)\times \frac{1}{\sum^{m_0}_{i=1}\hat{\nu}(\mathrm{d}x,i)}\\
                &=\sum^{m_0-1}_{i=0}\int_{\mathbb{R}}[c(i)+[sr_1(i)+(1-s-l)r_2](x-K)^+-\Psi(a,i)-a\mu_C(i)]\nu(\mathrm{d}x,i) \text{ a.s.}
            \end{aligned}
        \end{equation}
        Hence, 
        \begin{equation}\label{pp}
            \mathbb{P}\left(
                \begin{aligned}
                &\lim_{T\to\infty}\frac{1}{T}\int^T_0[c(\alpha(t))+[s(X(t),\alpha(t))r_1(\alpha(t))+(1-s(X(t),\alpha(t))\\
                &-l(X(t),\alpha(t)))r_2](X(t)-K)^+-\Psi(a(X(t),\alpha(t)),\alpha(t))\\
                &-a(X(t),\alpha(t))\mu_C(\alpha(t))]\mathrm{d}t\\
                &=\sum^{m_0-1}_{i=0}\int_{\mathbb{R}}[c(i)+[sr_1(i)+(1-s-l)r_2](x-K)^+-\Psi(a,i)\\
                &-a\mu_C(i)]\nu(\mathrm{d}x,i)
            \end{aligned}
            \right)=1.
        \end{equation}
        Note that \eqref{pp} holds for any $(x,i)\in G\times\mathcal{M}$, 
        and $\zeta(\cdot,\cdot)$ is the stationary density associated with the stationary distribution 
        $\nu(\cdot,\cdot)$, then \eqref{pp} can be written as \eqref{pb}.
        
  \end{proof}
  
  \section{Dynamic programming equation}

      Now, we have the stationary distribution $\nu(\cdot,\cdot)$. However, 
      it is generally difficult to approximate the invariant measure. To 
      get the optimal ergodic control of long-run average reward, we will refer to 
      the dynamic programming equation in \eqref{HJB}. We consider a two-component 
      Markov chain to approximate the state process, where one component corresponds to the 
      discretized surplus state and the other to the regime. Then we will rewrite \eqref{HJB} as 
      a dynamic programming equation with a Markov chain with transition probabilities. 

      Before deriving the dynamic programming equation, a review of key Markov chain concepts is required.
      By using the ergodic theorem for Markov chains in \cite{bertsekas1987dynamic,kushner1971introduction}, we 
      can find an auxiliary function $W(x,i,u)$ such that the pair $(W(x,i,u),\gamma(u))$ satisfies 
      \begin{align}
          W(x,i,u)=\sum_{y,j}p((x,i),(y,j)\vert u)W(y,j,u)+f(x,i,u)-\gamma(u),
      \end{align}
      for each feedback control $u(\cdot)$, where $p((x,i),(y,j)\vert u)$ is the transition 
      probability from a state $(x,i)$ to another state $(y,j)$ under the control $u(\cdot)$. 
      Define $\bar{\gamma}=\max_u\gamma(u)$, where $u(\cdot)\in\mathcal{U}$. Then 
      there is an auxiliary function $V(x,i)$ such that the pair $(V(x,i),\bar{\gamma})$ satisfies 
      the dynamic programming equation: 
      \begin{align}\label{HJBV}
          V(x,i)=\max_{u\in\mathcal{U}}\left\{\sum_{y,j}p((x,i),(y,j)\vert u)V(y,j)+f(x,i,u)-\bar{\gamma}\right\}.
      \end{align}
      To keep $V(x,i)$ from blowing up, \eqref{HJBV} can be rewritten in a centred form as 
      \begin{align}
          V(x,i)=\max_{u\in\mathcal{U}}\left\{\sum_{y,j}p((x,i),(y,j)\vert u)\tilde{V}(y,j)+f(x,i,u)\right\},
      \end{align}
      where 
      \begin{align}
          \tilde{V}(y,j)=V(y,j)-V(x_0,i).
      \end{align}
      $x_0$ is determined such that $\bar{\gamma}=V(x_0,i)$.

      It is always necessary to assume a compact state space for numerical analysis. 
      For the convenience of calculation and considering the positive recurrence of the 
      surplus process and the objective function, it is very natural to set a sufficiently 
      large reflecting boundary value $B$. Thus, for the boundaries, $V(x,i)$ follows 
      \begin{align}
          V_x(x,i)=0.
      \end{align}

      \section{Numerical approximation}
      We want to design a numerical scheme to 
      approximate $\bar{\gamma}$ in \eqref{optimal_value}. We will make an approximating 
      a Markov chain in the state space, and give the discrete form of the dynamic programming 
      equation and the transition probability of the approximate Markov chain.

      \subsection{Approximating Markov chain}
      A locally consistent discrete-time Markov chain will be constructed to approximate the 
      controlled regime-switching diffusion system, which will be 
      locally consistent with \eqref{surplus process}. By using the method in \cite{kushner2001numerical} and \cite{JinYangYin2018}, our approximating Markov chain 
      will include two components: one component is to approximate the diffusive part, and the 
      other describes the market regimes.

     Assuming the step size  $h>0$, we define $S^{\prime}_h=\{x:x=kh,k=0,\pm 1,\pm 2,\cdots\}$ and $G_h=(-(B+h),B+h)$ such that $S_h=S^{\prime}_h\cap G_h$. Here 
       $B$ is a very large value representing the bound for 
      computation purpose, and $B$ can be an integer multiple of $h$ without 
      losing generality. Suppose $\{(\xi^h_n,\alpha^h_n),n<\infty\}$ is a controlled discrete time 
      Markov chain on $S_h\times \mathcal{M}$, and $p^h((x,i),(y,j)\vert u^h)$ is the 
      transition probability from one state $(x,i)$ to another state $(y,j)$ under the control 
      $u^h$. Next, we will define $p^h$ so that the evolution of the constructed Markov chain can 
      approximate the local behavior of the controlled regime-switching diffusion 
      process \eqref{surplus process}. 

      At a discrete time $n$, we can not only carry out regular control, such as reinsurance, 
      investment and dividend, but also reflect on the boundary. 
      Let $\Delta \xi^h_n=\xi^h_{n+1}-\xi^h_n$, then 
      \begin{equation}\label{Deltaxi}
          \begin{aligned}
              \Delta\xi^h_n&=\Delta\xi^h_n\boldsymbol{I}_{\{\text{regular control at }n\}}
              +\Delta\xi^h_n\boldsymbol{I}_{\{\text{reflection step on the left at }n\}}\\
              &+\Delta\xi^h_n\boldsymbol{I}_{\{\text{reflection step on the right at }n\}}.
          \end{aligned}
      \end{equation}
      Note that only one term in \eqref{Deltaxi} is non-zero. Suppose a sequence of 
      control actions is given by $\{I^h_n:n=0,1,\cdots\}$, 
      where $I^h_n=0,1,\text{ or } 2$ corresponds to selecting a regular control, 
      reflection at the left boundary, or reflection at the right boundary, respectively, at time $n$. 
      If $I^h_n=0$, then we denote the regular control for the chain at time $n$ 
      by the random variable $u^h_n\subset U$ with $U=[M_a,1]\times[0,M_s]\times[M_l,1]$.  
      Let $\tilde{\Delta}t^h(\cdot,\cdot,\cdot)>0$ be the interpolation interval on 
      $S_h\times \mathcal{M}\times U$ and 
      \begin{equation}
          \begin{aligned}
              \inf_{x,i,u}\tilde{\Delta}t^h(x,i,u)>0, \forall h>0,\ \ 
              \lim_{h\to 0}\sup_{x,i,u}\tilde{\Delta}t^h(x,i,u)\to 0.
          \end{aligned}
      \end{equation}
      If $I^h_n=1$, or $\xi^h_n=-(B+h)$, it is necessary to apply the reflection step on the 
      left boundary. Left reflection makes the state from $-(B+h)$ to $-B$. Define $\Delta z^h_n$ 
      as the left reflection step size at time $n$, then $\Delta \xi^h_n=\Delta z^h_n=h$. 
      Similarly, when $I^h_n=2$, or $\xi^h_n=B+h$, we apply the reflection step on the right 
      boundary. Let $\Delta g^h_n$ be the right reflection step size at time $n$, 
      then $\Delta \xi^h_n=-\Delta g^h_n=-h$. 
      
      Define $\mathbb{E}^{u,h,0}_{x,i,n}$, $\mathrm{Var}^{u,h,0}_{x,i,n}$ and $\mathbb{P}^{u,h,0}_{x,i,n}$ 
      as the conditional expectation, variance, and marginal probability 
      given $\{\xi^h_k,\alpha^h_k, u^h_k,I^h_k,k\le n, \xi^h_n=x,\alpha^h_n=i,I^h_n=0,u^h_n=(a,s,l)\}$, 
      respectively. The sequence $\{(\xi^h_n,\alpha^h_n)\}$ is said to be locally consistent, if it 
      satisfies 
      \begin{equation}
          \begin{aligned}
              \mathbb{E}^{u,h,0}_{x,i,n}[\Delta\xi^h_n]&=[c(i)-\Psi(a,i)-a\mu_C(i)+sr_1(i)(x-K)^+\\
              &\ \ +(1-s-l)r_2(x-K)^+-l(x-K)^+]\tilde{\Delta}t^h(x,i,u)
              +o(\tilde{\Delta}t^h(x,i,u))\\
              \mathrm{Var}^{u,h,0}_{x,i,n}[\Delta\xi^h_n]&=[a^2\sigma^2_C(i)+s^2\sigma^2_A(i)((x-K)^+)^2]\tilde{\Delta}t^h(x,i,u)
              +o(\tilde{\Delta}t^h(x,i,u))\\
              \mathbb{P}^{u,h,0}_{x,i,n}(\alpha^h_{n+1}=j)&=q_{ij}\tilde{\Delta}t^h(x,i,u)+o(\tilde{\Delta}t^h(x,i,u)), \forall j\neq i,\\
              \mathbb{P}^{u,h,0}_{x,i,n}(\alpha^h_{n+1}=i)&=1+q_{ii}\tilde{\Delta}t^h(x,i,u)+o(\tilde{\Delta}t^h(x,i,u)),\\
              \sup_{n,\omega\in \Omega}\vert \Delta \xi^h_n\vert &\to 0 \text{ as }h\to 0.
          \end{aligned}
      \end{equation}
      When $I^h_n=1$ or $2$, we let the reflections be instantaneous. That is, 
      for any $(x,i,u,\bar{i})\in S_h\times \mathcal{M}\times U\times\{0,1,2\}$, 
      the interpolation interval is 
      \begin{align}\label{interval}
          \Delta t^h(x,i,u,\bar{i})=\tilde{\Delta}t^h(x,i,u)\boldsymbol{I}_{\{\bar{i}=0\}}.
      \end{align}
      The control strategy $u^h=\{u^h_n,n<\infty\}$ for the chain is said to be admissible if the chain has the 
      Markov property under that strategy. In particular, the strategy $u^h$ is admissible if $u^h_n$ is  
      $\sigma\{(\xi^h_0,\alpha^h_0),\cdots,(\xi^h_n,\alpha^h_n),u^h_0,\cdots,u^h_{n-1}\}$-adapted 
      and for any $\mathcal{G}\in\mathcal{B}(S_h\times\mathcal{M})$, we have 
      \begin{equation}
        \begin{aligned}
          &\mathbb{P}\left\{(\xi^h_{n+1},\alpha^h_{n+1})\in\mathcal{G}\vert\sigma\{(\xi^h_k,\alpha^h_k),u^h_k,k\le n\}\right\}
          =p^h((\xi^h_n,\alpha^h_n),\mathcal{G}\vert u^h_n),\\
          &\mathbb{P}\left\{(\xi^h_{n+1},\alpha^h_{n+1})=(-B,i)\vert (\xi^h_n,\alpha^h_n)=(-(B+h),i),
          \sigma\{(\xi^h_k,\alpha^h_k),u^h_k,k\le n\}\right\}=1\\
          &\mathbb{P}\left\{(\xi^h_{n+1},\alpha^h_{n+1})=(B,i)\vert (\xi^h_n,\alpha^h_n)=(B+h,i),
          \sigma\{(\xi^h_k,\alpha^h_k),u^h_k,k\le n\}\right\}=1.
        \end{aligned}
    \end{equation}
        Let 
        \begin{equation}
            \begin{aligned}
                t^h_0=0,t^h_n=\sum^{n-1}_{k=0}\Delta t^h(\xi^h_k,\alpha^h_k,u^h_k,I^h_k),
                \Delta t^h_k=\Delta t^h(\xi^h_k,\alpha^h_k,u^h_k,I^h_k), n^h(t)=\max\{n:t^h_n\le t\}.
            \end{aligned}
        \end{equation}
        Then the piecewise constant interpolations $(\xi^h(\cdot),\alpha^h(\cdot))$, $u^h(\cdot)$, 
        $z^h(\cdot)$ and $g^h(\cdot)$ for $t\in[t^h_n,t^h_{n+1})$ are defined as 
        \begin{equation}\label{intervalmarkov}
            \begin{aligned}
                &\xi^h(t)=\xi^h_n,\alpha^h(t)=\alpha^h_n,u^h(t)=u^h_n=u(\xi^h_n,\alpha^h_n),\\
                &z^h(t)=\sum_{k\le n^h(t)}\Delta z^h_k\boldsymbol{I}_{\{I^h_k=1\}},g^h(t)=\sum_{k\le n^h(t)}\Delta g^h_k\boldsymbol{I}_{\{I^h_k=2\}}.
            \end{aligned}
        \end{equation}
        Let $(\xi^h_0,\alpha^h_0)=(x,i)\in S_h\times \mathcal{M}$ and $u^h$ be an 
        admissible strategy. Thus, the return function for the controlled Markov chain follows: 
        \begin{align}
            J^h_B(x,i,u)=\liminf_n\frac{\mathbb{E}_{x,i}\sum^{n-1}_{k=1}f(\xi^h_k,\alpha^h_k,u^h_k)\Delta t^h_k}{\mathbb{E}_{x,i}\sum^{n-1}_{k=1}\Delta t^h_k},
        \end{align}
        which is analogous to \eqref{objective function} regarding the definition of interpolation 
        intervals in \eqref{interval}. Since $J^h_B(x,i,u)$ does not depend on the initial 
        condition $(x,i)$, we write it as $\gamma^h(u)$. Similarly, we denote 
        \begin{align}\label{gammah}
            \bar{\gamma}^h=\sup_{u^h \text{admissible}}\gamma^h(u).
        \end{align}
      Note that we are considering feedback controls $u(\cdot)$ here. 
      Similarly to $\nu$ in \eqref{stationarymeasure}, let $\nu^h(u)=\nu^h(x,i,u),(x,i)\in S_h\times \mathcal{M}$ 
      denote the associate invariant measure in the approximating space. Thus, $\gamma^h(u)$ 
      can be rewritten as 
      \begin{equation}\label{gammarewrite}
          \begin{aligned}
              \gamma^h(u)&=\liminf_n\frac{\mathbb{E}_{x,i}\sum^{n-1}_{k=1}f(\xi^h_k,\alpha^h_k,u^h_k)\Delta t^h_k}{\mathbb{E}_{x,i}\sum_{k=1}^{n-1}\Delta t^h_k}\\
              &=\frac{\sum_{x,i}f(x,i,u)\Delta t^h(x,i,u(x,i),0)\nu^h(x,i,u)}{\sum_{x,i}\Delta t^h(x,i,u(x,i),0)\nu^h(x,i,u)}.
          \end{aligned}
      \end{equation}
      Since the time interval of the approximating Markov chain $\Delta t^h(x,i,u(x,i),0)$ 
      depends on $(x,i)$ and $u$, the invariant measure for the approximating Markov chain needs 
      to consider the time spent on each state of the interpolated process. Thus, we define a 
      new measure $\omega^h(u)=\omega^h(x,i,u),(x,i)\in S_h\times \mathcal{M}$ such that 
      \begin{align}
          \omega^h(x,i,u)=\frac{\Delta t^h(x,i,u(x,i),0)\nu^h(x,i,u)}{\sum_{x,i}\Delta t^h(x,i,u(x,i),0)\nu^h(x,i,u)}.
      \end{align}
      Hence, $\gamma^h(u)$ can be written in a simple form as 
      \begin{align}\label{gammare2}
          \gamma^h(u)=\sum_xf(x,i,u)\omega^h(x,i,u).
      \end{align}
      Define $\mathbb{E}^u_{\omega^h(u)}$ be the expectation for the stationary process 
      under control $u(\cdot,\cdot)$. In view of \eqref{gammarewrite}, \eqref{gammare2} can also be written 
      as 
      \begin{align}\label{gammaexp}
          \gamma^h(u)=\mathbb{E}^u_{\omega^h(u)}\int^1_0f(\xi^h(s),\alpha^h(s),u^h(\xi^h(s),\alpha^h(s)))\mathrm{d}s.
      \end{align}
      Actually, it is much more difficult to calculate the invariant measure $\omega^h(u)$. 
      Through the iteration method, the convergence rate of computing the value of $\gamma^h(u)$ is much faster than 
      that of computing the invariant measure $\omega^h(u)$. Thus, we focus on the 
      convergence of the state process and objective functions rather than invariant measures. 
      Next, we will show that $V^h(x,i)$ satisfies the dynamic programming equation as 
      \begin{equation}\label{Valuefunction}
          \begin{aligned}
              V^h(x,i)=
              \begin{cases}
                  &\max_{u}\left[\sum_{y,j}p^h((x,i),(y,j)\vert u)V^h(y,j)
                  +(f(x,i,u)\right.\\
                  &\left.-\gamma^h)\Delta t^h(x,i,u,0)\right],
                  \text{ for }(x,i) \in S_h\times \mathcal{M},\\
                  &\max_{u}\left[\sum_{y,j}p^h((x,i),(y,j)\vert u)V^h(y,j)\right],
                  \text{ for }(x,i)\in \partial S_h\times \mathcal{M}.
              \end{cases}
          \end{aligned}
      \end{equation}

      \subsection{Discretization}
      
      Define the approximation to the first and the second derivatives of $V(\cdot,i)$ by 
      the finite difference method in using the step size $h>0$ as 
      \begin{equation}
          \begin{aligned}
              &V(x,i)\to V^h(x,i)\\
              &V_x(x,i)\to\frac{V^h(x+h,i)-V^h(x,i)}{h} \text{ for } b(x,i,u)>0,\\
              &V_x(x,i)\to\frac{V^h(x,i)-V^h(x-h,i)}{h} \text{ for } b(x,i,u)<0,\\
              &V_{xx}(x,i)\to\frac{V^h(x+h,i)-2V^h(x,i)+V^h(x-h,i)}{h^2}.
          \end{aligned}
      \end{equation}
      It leads to, $\forall x\in S_h,i\in\mathcal{M}$,
      \begin{equation}\label{HJBVh}
          \max_{u\in U}\left\{
          \begin{aligned}
              &\frac{V^h(x+h,i)-V^h(x,i)}{h}[b(x,i,u)]^+
              -\frac{V^h(x,i)-V^h(x-h,i)}{h}[b(x,i,u)]^-\\
              &+\frac{1}{2}\sigma^2(x,i,u)\cdot\frac{V^h(x+h,i)-2V^h(x,i)+V^h(x-h,i)}{h^2}
              +\sum_{j}V^h(x,\cdot)q_{ij}\\
              &+f(x,i,u)-\bar{\gamma}
          \end{aligned}
          \right\}=0.
      \end{equation}
      For the reflecting boundaries, we set 
      \begin{align}\label{Vxboundary}
          V_x(x,i)\to\frac{V^h(x,i)-V^h(x-h,i)}{h}.
      \end{align}
      That is, the process will be reflected left into the domain of the $x$. 
      Comparing \eqref{HJBVh} and \eqref{Vxboundary} with \eqref{Valuefunction}, we 
      obtain the transition probabilities of $V^h(x,i)$ in the interior of the domain as follows:
      \begin{equation}\label{probabilities}
          \begin{aligned}
              &p^h((x,i),(x+h,i)\vert u)=\frac{\sigma^2(x,i,u)^2/2+h[b(x,i,u)]^+}{D},\\
              &p^h((x,i),(x-h,i)\vert u)=\frac{\sigma^2(x,i,u)^2/2+h[b(x,i,u)]^-}{D},\\
              &p^h((x,i),(x,j)\vert u)=\frac{q_{ij}h^2}{D}, \text{ for }i\neq j,\\
              &p^h(\cdot)=0, \text{ otherwise,}\\
              &\Delta t^h(x,i,u,0)=\frac{h^2}{D},
          \end{aligned}
      \end{equation}
      with 
      \begin{align}
          D=\sigma(x,i,u)^2+h\vert b(x,i,u)\vert-h^2q_{ii}
      \end{align}
      being well defined. We also have the transition probability of $V^h(x,i)$ on the 
      boundaries comparing with \eqref{Valuefunction} as
      \begin{equation}
          \begin{aligned}
              &p^h((x,i),(x+h,i)\vert u)=1,\text{ for }x=-(B+h),\\
              &p^h((x,i),(x-h,i)\vert u)=1,\text{ for }x=B+h.
          \end{aligned}
      \end{equation}

  \section{Convergence of numerical approximation}
      In this section, we investigate the asymptotic properties of the approximating Markov 
      chain using the weak convergence techniques developed in \cite{JinYangYin2018}.

      \subsection{Interpolation and rescaling}
      We have finished the piecewise constant interpolation and chosen an appropriate interpolation 
      interval level by the approximating Markov chain defined in the last section. 
      In view of \eqref{intervalmarkov}, the continuous-time interpolations $(\xi^h(\cdot),\alpha^h(\cdot))$, 
      $u^h(\cdot,\cdot)$, $z^h(\cdot)$, and $g^h(\cdot)$ are defined. Moreover, let $\mathcal{U}^h$ 
      be the collection of strategies, which is determined by a sequence of measurable 
      functions $F^h_n(\cdot)$. For $u^h_n\in\mathcal{U}^h$, 
      \begin{align}\label{Udefine}
          u^h_n=F^h_n(\xi^h_k,\alpha^h_k,k\le n;u^h_k,k\le n).
      \end{align}
      We define $\mathcal{D}^h_t$ as the smallest $\sigma$-algebra generated by $\left\{\xi^h(s),
      \alpha^h(s),u^h(s), g^h(s),z^h(s),s\le t\right\}$. $\mathcal{U}^h$ defined 
      by \eqref{Udefine} is equivalent 
      to the set of all piecewise constant admissible strategies with 
      respect to $\mathcal{D}^h_t$. 

      Combining the notations of the regular controls, interpolations, and reflection steps, we can obtain $\xi^h(t)$ from \eqref{Deltaxi} as follows.
      \begin{equation}\label{xih(t)}
          \begin{aligned}
              \xi^h(t)&=x+\sum^{n-1}_{k=0}\left[\mathbb{E}^h_k\Delta \xi^h_k+(\Delta\xi^h_k-\mathbb{E}^h_k\Delta\xi^h_k)\right]\\
              &=x+\sum^{n-1}_{k=0}b(\xi^h_k,\alpha^h_k,u^h_k)\Delta t^h(\xi^h_k,\alpha^h_k,u^h_k,0)
              +\sum^{n-1}_{k=0}(\Delta\xi^h_k-\mathbb{E}^h_k\Delta\xi^h_k)+\varepsilon^h(t)\\
              &=x+B^h(t)+M^h(t)+\varepsilon^h(t),
          \end{aligned}
      \end{equation}
      where 
        \begin{equation}
            \begin{aligned}
            B^h(t)=\sum^{n-1}_{k=0}b(\xi^h_k,\alpha^h_k,u^h_k)\Delta t^h(\xi^h_k,\alpha^h_k,u^h_k,0),
            M^h(t)=\sum^{n-1}_{k=0}(\Delta\xi^h_k-\mathbb{E}^h_k\Delta\xi^h_k),
            \end{aligned}
        \end{equation}
      and $\varepsilon^h(t)$ is a negligible error satisfying 
      \begin{align}
          \lim_{h\to\infty}\sup_{0\le t\le T}E\vert\varepsilon^h(t)\vert^2\to 0,\text{ for any } 0<T<\infty.
      \end{align}
      Moreover, $M^h(t)$ is a martingale w.r.t. $\mathcal{D}^h_t$, and its discontinuity 
      goes to zero as $h\to\infty$. Now, we attempt to represent $M^h(t)$ similarly to the diffusion 
      term in \eqref{surplus process}. Define $W^h(\cdot)$ as 
      \begin{equation}\label{wdefine}
          \begin{aligned}
              W^h(t)&=\sum^{n-1}_{k=0}\frac{\Delta\xi^h_k-\mathbb{E}^h_k\Delta\xi^h_k}{\sigma(\xi^h_k,\alpha^h_k,u^h_k)}
              =\int^t_0\sigma^{-1}(\xi^h(s),\alpha^h(s),u^h(s))\mathrm{d}M^h(s).
          \end{aligned}
      \end{equation}
      \eqref{xih(t)} can be expressed  as 
      \begin{equation}\label{xihint}
          \begin{aligned}
              \xi^h(t)=x+\int^t_0b(\xi^h(s),\alpha^h(s),u^h(s))\mathrm{d}s
              +\int^t_0\sigma(\xi^h(s),\alpha^h(s),u^h(s))\mathrm{d}W^h(s)
              +\varepsilon^h(t).
          \end{aligned}
      \end{equation}
      We are going to do the rescaling next. The technique of time-rescaling is to 
      'stretch out' the control and state processes to make them smoother. Then we can 
      prove the tightness of $g^h(\cdot)$ and $z^h(\cdot)$. Define $\Delta \hat{t}^h_n$ by 
      \begin{equation}
          \begin{aligned}
              \Delta \hat{t}^h_n=
              \begin{cases}
                  \Delta t^h_n \text{ for a diffusion on step }n,\\
                  \vert\Delta z^h_n\vert=h \text{ for a left reflection on step }n,\\
                  \vert \Delta g^h_n\vert =h \text{ for a right reflection on step }n.
              \end{cases}
          \end{aligned}
      \end{equation}
      Define $\hat{T}^h(\cdot)$ by 
      \begin{align}
          \hat{T}^h(t)=\sum^{n-1}_{i=0}\Delta t^h_i=t^h_n, \text{ for }t\in[\hat{t}^h_n,\hat{t}^h_{n+1}).
      \end{align}
      When a regular control is exerted, the $\hat{T}^h(\cdot)$ will increase with unit slope. 
      In addition, we define the rescaled and interpolated process 
      $\hat{\xi}^h(t)=\xi^h(\hat{T}^h(t))$. $\hat{\alpha}^h(t)$, $\hat{u}^h(t)$, $\hat{g}^h(t)$ 
      and $\hat{z}^h(t)$ are defined similarly. The time scale is stretched out by $h$ at the left and right reflection 
      steps. Then the stretched process can be written as follows.
      \begin{equation}\label{xihrescaled}
          \begin{aligned}
              \hat{\xi}^h(t)=x+\int^t_0b(\hat{\xi}^h(s),\hat{\alpha}^h(s),\hat{u}^h(s))\mathrm{d}s
              +\int^t_0\sigma(\hat{\xi}^h(s),\hat{\alpha}^h(s),\hat{u}^h(s))\mathrm{d}W^h(s)
              +\varepsilon^h(t).
          \end{aligned}
      \end{equation}

      \subsection{Relaxed controls}
      In order to facilitate the proof of weak convergence, we introduce the relaxed control 
      representation. For more details, see \cite{kushner2001numerical} and \cite{JinYangYin2018}.

      Let $\mathcal{B}(U\times[0,\infty))$ be the $\sigma$-algebra of Borel subsets of 
      $U\times[0,\infty)$. An admissible relaxed control $m(\cdot)$ is a meausre on 
      $\mathcal{B}(U\times[0,\infty))$ such that $m(U\times[0,t])=1$ for each $t\ge0$. 
      Given a relaxed control $m(\cdot)$, there is an $m_t(\cdot)$ such that 
      $m(\mathrm{d}\phi\mathrm{d}t)=m_t(\mathrm{d}\phi)\mathrm{d}t$. Given a 
      relaxed control $m(\cdot)$ 
      of $u^h(\cdot)$, we define the derivative $m_t(\cdot)$ such that 
      \begin{align}
          m^h(H)=\int_{U\times[0,\infty)}\boldsymbol{I}_{\{(u^h,t)\in H\}}m_t(\mathrm{d}\phi)\mathrm{d}t
      \end{align}
      for all $H\in\mathcal{B}(U\times [0,\infty])$. For each $t$, $m_t(\cdot)$ is a 
      measure on $\mathcal{B}(U)$ satisfying $m_t(U)=1$. Then $m_t(\cdot)$ can 
      be defined as the left-hand derivative for $t>0$,
      \begin{align}
          m_t(O)=\lim_{\delta \to 0}\frac{m(O\times[t-\delta,t])}{\delta},\forall O\in\mathcal{B}(U).
      \end{align}
      The relaxed control representation $m^h(\cdot)$ of $u^h(\cdot)$ can be defined by 
      \begin{align}
          m^h_t(O)=\boldsymbol{I}_{\{u^h(t)\in O\}},\forall O\in\mathcal{B}(U).
      \end{align}
      Denote by $\mathcal{F}^h_t$ a filtration. $\mathcal{F}^h_t$ is the minimal $\sigma$-algebra 
      that measures 
      \begin{align}
          \{\xi^h(s),\alpha^h(\cdot),m^h_s(\cdot),W^h(s),z^h(s),g^h(s),s\le t\}.
      \end{align}
      Let $\varGamma^h$ be the set of admissible relaxed controls $m^h(\cdot)$ with respect 
      to $(\alpha^h(\cdot),W^h(\cdot))$ such that $m^h_t(\cdot)$ is a fixed probability 
      measure in the interval $[t^h_n,t^h_{n+1})$ given $\mathcal{F}^h_t$. Then $\varGamma^h$ is 
      a larger control space containing $\mathcal{U}^h$. Combining with the stretched out time scale, 
      we denote the rescaled relax control as $m_{\hat{T}^h(t)}(\mathrm{d}\phi)$. Define $M_t(O)$ 
      and $M^h_t(\mathrm{d}\phi)$ by 
        \begin{equation}
            \begin{aligned}
                M_t(O)\mathrm{d}t=\mathrm{d}W(t)\boldsymbol{I}_{u(t)\in O},\forall O\in\mathcal{B}(U),\ \ 
                M^h_t(\mathrm{d}\phi)\mathrm{d}t=\mathrm{d}W^h(t)\boldsymbol{I}_{u^h(t)\in\mathcal{U}}.
            \end{aligned}
        \end{equation}
      Likely, we set 
      \begin{align}
          \hat{M}^h_{\hat{T}^h(t)}(\mathrm{d}\phi)\mathrm{d}\hat{T}^h(t)=\mathrm{d}\hat{W}^h(\hat{T}^h(t))\boldsymbol{I}_{u^h(\hat{T}^h(t))\in\mathcal{U}}.
      \end{align}
      Taking into account the relaxed controls, we rewrite \eqref{xihint}, \eqref{xihrescaled}, 
      and \eqref{gammah} as
      \begin{equation}\label{xihMs}
          \begin{aligned}
              \xi^h(t)&=x+\int^t_0\int_{\mathcal{U}}b(\xi^h(s),\alpha^h(s),\phi)m^h_s(\mathrm{d}\phi)\mathrm{d}s\\
              &+\int^t_0\int_{\mathcal{U}}\sigma(\xi^h(s),\alpha^h(s),\phi)M^h_s(\mathrm{d}\phi)\mathrm{d}s+\varepsilon^h(t),
          \end{aligned}
      \end{equation}
      \begin{equation}\label{hatxi}
          \begin{aligned}
              \hat{\xi}^h(t)&=x+\int^t_0\int_{\mathcal{U}}b(\hat{\xi}^h(s),\hat{\alpha}^h(s),\phi)\hat{m}^h_{\hat{T}^h(s)}(\mathrm{d}\phi)\mathrm{d}\hat{T}^h(s)\\
              &+\int^t_0\int_{\mathcal{U}}\sigma(\hat{\xi}^h(s),\hat{\alpha}^h(s),\phi)\hat{M}^h_{\hat{T}^h(s)}(\mathrm{d}\phi)\mathrm{d}\hat{T}^h(s)
              +\varepsilon^h(t),
          \end{aligned}
      \end{equation}
      and 
      \begin{align}
          \bar{\gamma}^h=\sup_{m^h\in\varGamma^h}\gamma^h(m^h).
      \end{align}
      Next, we define the existence and uniqueness of the solution in the weak sense. 
      \begin{definition}
          By a weak solution of \eqref{xihMs}, we mean that there exists a probability 
          space $(\Omega,\mathcal{F},\{\mathcal{F}_t\},P)$, and the sequence of processes 
          $(x(\cdot),\alpha(\cdot), m(\cdot),W(\cdot))$ such that $W(\cdot)$ is a standard 
          $\mathcal{F}_t$-Wiener process, $\alpha(\cdot)$ is a continuous time Markov chain, 
          $m(\cdot)$ is admissible with respect to $x(\cdot)$ is $\mathcal{F}^t$-adapted, and 
          \eqref{xihMs} is satisfied. For an initial condition $(x,i)$, we say that the 
          probability law of the admissible process $(x(\cdot),\alpha(\cdot),m(\cdot),W(\cdot))$ 
          determines the probability law of solution $(x(\cdot),\alpha(\cdot),m(\cdot),W(\cdot))$ 
          to \eqref{xihMs} by the weak sense uniqueness, irrespective of the probability space.
      \end{definition}
      
      In addition, we have one more assumption.
      \begin{assumption}
          Let $u(\cdot)=(a(\cdot),s(\cdot),l(\cdot))$ be an admissible ordinary control with respect to $W(\cdot)$ and $\alpha(\cdot)$. 
          Assume that $u(\cdot)$ is piecewise constant and takes values in a finite set. For each 
          initial condition, there exists a solution, which is unique in the weak sense, to \eqref{xihMs} 
          where $m(\cdot)$ is the relaxed control representation of $u(\cdot)$.
      \end{assumption}

      \subsection{Convergence of a sequence of surplus processes}
      In this section, we deal with the convergence of the approximation sequence to the regime-switching 
      process and the surplus process. We will derive one lemma and three theorems. 
      \begin{lemma}\label{wapprox}
          Using the transition probabilities $\{p^h(\cdot)\}$ defined in \eqref{probabilities}, 
          the interpolated process of the constructed Markov chain $\{\hat{\alpha}^h(\cdot)\}$ converges 
          weakly to $\hat{\alpha}(\cdot)$, the Markov chain with generator $Q=(q_{ij})$.
      \end{lemma}

      \begin{proof}
        Similar to Theorem 3.1 in \cite{YinZhangBadowski2003}, we can show that 
        \begin{equation}\label{alpha_h_convergence}
            \begin{aligned}
                \left\vert\mathbb{E}\left[(\alpha^h(t+s)-\alpha^h(t))^2\vert\mathcal{F}^h_t\right]\right\vert\le \bar{\gamma}^h(s),\ \ 
                \lim_{s\to0}\limsup_{h\to0}\mathbb{E}(\bar{\gamma}^h(s))=0,
            \end{aligned}
        \end{equation}
        where $\bar{\gamma}^h(s)\ge0$ is $\mathcal{F}^h_t$-measurable. On the other hand, 
        due to the definition of $\hat{\alpha}^h(\cdot)$, we have 
        \begin{equation}\label{alpha_hat_h}
            \begin{aligned}
                \left\vert\mathbb{E}\left[(\hat{\alpha}^h(t+s)-\hat{\alpha}^h(t))^2\right]\vert\mathcal{F}^h_t\right\vert
                \le\left\vert\mathbb{E}\left[(\alpha^h(t+s)-\alpha^h(t))^2\right]\vert\mathcal{F}^h_t\right\vert
                \le\bar{\gamma}^h(s).
            \end{aligned}
        \end{equation}
        Combining \eqref{alpha_h_convergence} and \eqref{alpha_hat_h}, we obtain that $\hat{\alpha}^h(\cdot)$ 
        is tight. Furthermore, it can be shown that the constructed Markov chain $\{\hat{\alpha}^h(\cdot)\}$ 
        converges weakly to $\hat{\alpha}(\cdot)$.
      \end{proof}

      \begin{theorem}\label{tight}
          Let the approximating chain $\{\xi^h_n,\alpha^h_n,n<\infty\}$ constructed with 
          transition probabilities defined in \eqref{probabilities} be locally consistent 
          with \eqref{surplus process}, $m^h(\cdot)$ be the relaxed control representation of 
          $\{u^h_n,n<\infty\}$, $(\xi^h(\cdot),\alpha^h(\cdot))$ be the continuous-time interpolation 
          defined in \eqref{intervalmarkov}, and $\{\hat{\xi}^h(\cdot),\hat{\alpha}^h(\cdot),\hat{m}^h(\cdot),
          \hat{W}^h(\cdot),\hat{z}^h(\cdot),\hat{g}^h(\cdot),\hat{T}^h(\cdot)\}$ be the 
          corresponding rescaled processes. Then $\{\hat{\xi}^h(\cdot),\hat{\alpha}^h(\cdot),\hat{m}^h(\cdot),
          \hat{W}^h(\cdot),\hat{z}^h(\cdot),\hat{g}^h(\cdot),\hat{T}^h(\cdot)\}$ is tight.
      \end{theorem}

      \begin{proof}
        In view of Lemma \ref{wapprox}, $\{\hat{\alpha}^h(\cdot)\}$ is tight. Since its 
        range space is compact, the sequence $\{\hat{m}^h(\cdot)\}$ is tight.  
        Recall $W^h(\cdot)$ in \eqref{wdefine}, i.e.
        \begin{align}
            W^h(t)=\int^t_0\sigma^{-1}(\xi^h(s),\alpha^h(s),u^h(s))\mathrm{d}M^h(s).
        \end{align}
        Its increment over $\Delta t^h_k$ is 
        \begin{align}
            \Delta W^h_k=W^h(t_{k+1})-W^h(t_k)=\frac{\Delta M^h_k}{\sigma(\xi^h_k,\alpha^h_k,u^h_k)}.
        \end{align}
        The term $\Delta M^h_k= \Delta \xi^h_k-\mathbb{E}^h_k\Delta \xi^h_k$ is a martingale difference. 
        By local consistency of the discretization, 
        \begin{align}
            \mathbb{E}[(\Delta M^h)^2]=\sigma^2(\xi^h_k,\alpha^h_k,u^h_k)\Delta t^h_k+o(\Delta t^h).
        \end{align}
        Thus, $\mathbb{E}[(\Delta W^h_k)^2]=\Delta t^h_k+o(\Delta t^h)$. 
        Let $T<\infty$, and let $\tau_h$ be an $\mathcal{F}_t$-stopping time which 
        is not larger than $T$. The interval $[\tau_h,\tau_h+\delta]$ contains $N$ time steps, i.e. $\sum^{n+N-1}_{k=n}\Delta t^h_k
        =\delta +o(\delta)$. Then, we have 
        \begin{align}
            \mathbb{E}^{u^h_k}_{\tau_h}\left[(W^h(\tau_h+\delta)-W^h(\tau_h))^2\right]=\delta +\varepsilon_h,
        \end{align}
        where $\varepsilon_h\to 0$ uniformly in $\tau_h$. By Chebyshev's inequality, we have 
        \begin{align}
            \mathbb{P}(\vert W^h(\tau_h+\delta)-W^h(\tau_h)\vert >\epsilon)\le \frac{\mathbb{E}[(W^h(\tau_h+\delta)-W^h(\tau_h))^2]}{\epsilon^2}=\frac{\delta+\vert\varepsilon_h\vert}{\epsilon^2},
        \end{align}
        for all $\epsilon>0$. Taking $\limsup_{h\to 0}$ followed by $\lim_{\delta\to 0}$ yield the tightness 
        of $\{W^h(\cdot)\}$ by Aldous tightness criterion. 
        Similar to the argument of $\alpha^h(\cdot)$, the tightness of $\hat{W}^h(\cdot)$ is obtained. 
        Furthermore, following the definition of 'stretched out' timescale, 
        \begin{equation}
            \begin{aligned}
                \vert\hat{z}^h(\tau_h+\delta)-\hat{z}^h(\tau_h)\vert \le \vert\delta\vert + O(h),\ \ 
                \vert\hat{g}^h(\tau_h+\delta)-\hat{g}^h(\tau_h)\vert \le \vert\delta\vert + O(h).
            \end{aligned}
        \end{equation}
        Thus, $\{\hat{z}^h(\cdot),\hat{g}^h(\cdot)\}$ is tight. These results and the boundedness 
        of $b(\cdot)$ imply the tightness of $\{\xi^h(\cdot)\}$. Thus, $\{\hat{\xi}^h(\cdot), 
        \hat{\alpha}^h(\cdot),\hat{u}^h(\cdot),\hat{W}^h(\cdot),\hat{z}^h(\cdot),\hat{g}^h(\cdot),
        \hat{T}^h(\cdot)\}$ is tight.
        Since $\{\hat{\xi}^h(\cdot),\hat{\alpha}^h(\cdot),\hat{m}^h(\cdot),\hat{W}^h(\cdot),
        \hat{z}^h(\cdot),\hat{g}^h(\cdot),\hat{T}^h(\cdot)\}$ is tight, a weakly convergent 
        subsequence can be extracted. For simplicity, still index the subsequence by $h$.
        Denoted the limit by $\{\hat{x}(\cdot),\hat{\alpha}(\cdot),
        \hat{m}(\cdot),\hat{W}(\cdot),\hat{z}(\cdot),\hat{g}(\cdot),\hat{T}(\cdot)\}$, 
        whose paths are continuous w.p.1.

      \end{proof}

      \begin{theorem}\label{limitx}
          Under the conditions of Theorem \ref{tight}, let $\{\hat{x}(\cdot),\hat{\alpha}(\cdot),\hat{m}(\cdot),
          \hat{W}(\cdot),\hat{z}(\cdot),\\\hat{g}(\cdot),\hat{T}(\cdot)\}$ be the limit of 
          weakly convergent subsequence of  $\{\hat{\xi}^h(\cdot),\hat{\alpha}^h(\cdot),\hat{m}^h(\cdot),
          \hat{W}^h(\cdot),\\\hat{z}^h(\cdot),\hat{g}^h(\cdot),\hat{T}^h(\cdot)\}$. $W(\cdot)$ 
          is a standard $\mathcal{F}_t$-Wiener process and $m(\cdot)$ is admissible. Let 
          $\hat{\mathcal{F}}_t$ be the $\sigma$-algebra generated by $\{\hat{\xi}^h(\cdot),\hat{\alpha}^h(\cdot),
          \hat{m}^h(\cdot),\hat{W}^h(\cdot),\hat{z}^h(\cdot),\hat{g}^h(\cdot),\hat{T}^h(\cdot)\}$. 
          Then $\hat{W}(t)=W(\hat{T}(t))$ is an $\hat{\mathcal{F}}_t$-martingale with quadratic 
          variation $\hat{T}(t)$. The limit process follows:
          \begin{equation}\label{hatx}
              \begin{aligned}
                  \hat{x}(t)&=x+\int^t_0\int_{\mathcal{U}}b(\hat{x}(s),\hat{\alpha}(s),\phi)\hat{m}_{\hat{T}(s)}(\mathrm{d}\phi)\mathrm{d}\hat{T}(s)\\
                  &+\int^t_0\int_{\mathcal{U}}\sigma(\hat{x}(s),\hat{\alpha}(s),\phi)\hat{M}_{\hat{T}(s)}(\mathrm{d}\phi)\mathrm{d}\hat{T}(s).
              \end{aligned}
          \end{equation}
      \end{theorem}

      \begin{proof}
          For $\delta>0$, define the process $f(\cdot)$ by $f^{h,\delta}(t)=f^h(n\delta)$, 
          $t\in[n\delta,(n+1)\delta)$. Then, by the tightness of $\{\hat{\xi}^h(\cdot),\hat{\alpha}^h(\cdot)\}$, 
          \eqref{hatxi} can be rewritten as 
          \begin{equation}\label{hatxirw}
            \begin{aligned}
                \hat{\xi}^h(t)&=x+\int^t_0\int_{\mathcal{U}}b(\hat{\xi}^{h,\delta}(s),\hat{\alpha}^{h,\delta}(s),\phi)\hat{m}^h_{\hat{T}^h(s)}(\mathrm{d}\phi)\mathrm{d}\hat{T}^h(s)\\
                &+\int^t_0\int_{\mathcal{U}}\sigma(\hat{\xi}^{h,\delta}(s),\hat{\alpha}^{h,\delta}(s),\phi)\hat{M}^h_{\hat{T}^h(s)}(\mathrm{d}\phi)\mathrm{d}\hat{T}^h(s)
                +\varepsilon^{h,\delta}(t),
            \end{aligned}
          \end{equation}
          where
          \begin{align}
            \lim_{\delta\to0}\limsup_{h\to0}\mathbb{E}\vert\varepsilon^{h,\delta}(t)\vert=0.
          \end{align}
          If we can verify that $\hat{W}(\cdot)$ is an $\hat{\mathcal{F}}_t$-martingale, 
          then \eqref{hatx} could be obtained by taking limits in \eqref{hatxirw}. To 
          characterize $W(\cdot)$, let $t>0$, $\delta>0$, $p$, $\kappa$, and $\{t_k:k\le p\}$ 
          be given such that $t_k\le t\le t+\delta$ for all $k\le p$, $\varphi_j(\cdot)$ for 
          $j\le \kappa$ are real-valued and continuous functions on $U\times [0,\infty)$ 
          having compact supports for all $j\le \kappa$. Define 
          \begin{align}\label{varphim}
            (\varphi_j,\hat{m}^h)_t=\int^t_0\int_{\mathcal{U}}\varphi_j(\phi,s)\hat{m}^h_{\hat{T}(s)}(\mathrm{d}\phi)\mathrm{d}\hat{T}(s).
          \end{align}
          Let $\{\varGamma^\kappa_j:j\le\kappa\}$ be a sequence of non-decreasing partition of 
          $\varGamma$ such that $\Pi(\partial \varGamma^\kappa_j)=0$ for all $j$ and all $\kappa$, 
          where $\partial \varGamma^\kappa_j$ is the boundary of the set $\varGamma^\kappa_j$. 
          As $\kappa\to\infty$, let the diameters of the sets $\varGamma^\kappa_j$ go to zero. 
          Let $K(\cdot)$ be a real-valued and continuous function of its arguments with 
          compact support. In view of the definition of $\hat{W}(t)$, for each $i\in \mathcal{M}$, 
          we have 
            \begin{equation}
                \begin{aligned}
                &\mathbb{E}K(\hat{\xi}^h(t_k),\hat{\alpha}^h(t_k),\hat{W}^h(t_k),(\varphi_j,m^h)_{t_k},
                \hat{z}^h(t_k),\hat{g}^h(t_k),j\le\kappa,k\le p)\\
                &\times [\hat{W}^h(t+\delta)-\hat{W}^h(t)]=0.
                \end{aligned}
            \end{equation}
          By using the Skorokhod representation theorem and the dominated convergence 
          theorem, letting $h\to 0$, we obtain
          \begin{equation}\label{exhto0}
            \begin{aligned}
                &\mathbb{E}K(\hat{\xi}(t_k),\hat{\alpha}(t_k),\hat{W}(t_k),(\varphi_j,m)_{t_k},
                \hat{z}(t_k),\hat{g}(t_k),j\le\kappa,k\le p)\\
                &\times [\hat{W}(t+\delta)-\hat{W}(t)]=0.
            \end{aligned}
          \end{equation}
          Since $\hat{W}(\cdot)$ has continuous sample paths, \eqref{exhto0} implies 
          that $\hat{W}(\cdot)$ is a continuous $\mathcal{F}_t$-martingale. On the 
          other hand, since 
          \begin{equation}\label{seconddiff}
            \begin{aligned}
                \mathbb{E}\left[(\hat{W}^h(t+\delta))^2-(\hat{W}^h(t))^2\right]
                =\mathbb{E}\left[(\hat{W}^h(t+\delta)-\hat{W}^h(t))^2\right]
                =\hat{T}^h(t+\delta)-\hat{T}^h(t),
            \end{aligned}
          \end{equation}
          by using the Skorokhod representation and the dominant convergence 
          theorem together with \eqref{seconddiff}, letting $h\to 0$, we have 
          \begin{equation}\label{exhtW}
            \begin{aligned}
                &\mathbb{E}K(\hat{\xi}(t_k),\hat{\alpha}(t_k),\hat{W}(t_k),
                (\varphi_j,m)_{t_k},\hat{z}(t_k),\hat{g}(t_k),j\le\kappa,
                k\le p)\\
                &\times [\hat{W}^2(t+\delta)-\hat{W}^2(t)-(\hat{T}(t+\delta)-\hat{T}(t))]=0.
            \end{aligned}
          \end{equation}
          The quadratic variation of the martingale $\hat{W}(t)$ is 
          $\hat{T}(t)$, then $\hat{W}(\cdot)$ is an $\hat{\mathcal{F}}_t$-Wiener process.

          Letting $h\to 0$, by using the Skorokhod representation, we obtain 
          \begin{equation}
            \begin{aligned}
                \mathbb{E}\left\vert \int^t_0\int_{\mathcal{U}}b(\hat{\xi}^h(s),\hat{\alpha}^h(s),\phi)\hat{m}^h_{\hat{T}^h(s)}(\mathrm{d}\phi)\mathrm{d}\hat{T}^h(s)
                -\int^t_0\int_{\mathcal{U}}b(\hat{x}(s),\hat{\alpha}(s),\phi)\hat{m}^h_{\hat{T}(s)}(\mathrm{d}\phi)\mathrm{d}\hat{T}(s)\right\vert\to0
            \end{aligned}
          \end{equation}
          uniformly in $t$. On the other hand, $\{\hat{m}^h(\cdot)\}$ converges in the compact weak 
          topology, that is, for any bounded and continuous function $\varphi(\cdot)$ with 
          compact support, as $h\to 0$, 
          \begin{equation}
            \begin{aligned}
                \int^\infty_0\int_{\mathcal{U}}\varphi(\phi,s)\hat{m}^h_{\hat{T}^h(s)}(\mathrm{d}\phi)\mathrm{d}\hat{T}^h(s)\to\int^\infty_0\int_{\mathcal{U}}\varphi(\phi,s)\hat{m}_{\hat{T}(s)}(\mathrm{d}\phi)\mathrm{d}\hat{T}(s).
            \end{aligned}
          \end{equation}
          Again, the Skorokhod representation (with a slight abuse of notation) implies that, 
          as $h\to 0$,
          \begin{equation}
            \begin{aligned}
                \int^t_0\int_{\mathcal{U}}b(\hat{\xi}^h(s),\hat{\alpha}^h(s),\phi)\hat{m}^h_{\hat{T}^h(s)}(\mathrm{d}\phi)\mathrm{d}\hat{T}^h(s)
                \to\int^t_0\int_{\mathcal{U}}b(\hat{\xi}(s),\hat{\alpha}(s),\phi)\hat{m}_{\hat{T}(s)}(\mathrm{d}\phi)\mathrm{d}\hat{T}(s)
            \end{aligned}
          \end{equation}
          uniformly in $t$ on any bounded interval. 

          In view of \eqref{hatxirw}, since $\xi^{h,\delta}$ and $\alpha^{h,\delta}$ are 
          piecewise constant functions, as $h\to0$, we have
          \begin{equation}\label{sigmadelta}
            \begin{aligned}
                \int^t_0\int_{\mathcal{U}}\sigma(\hat{\xi}^{h,\delta}(s),\hat{\alpha}^{h,\delta}(s),\phi)\hat{M}_{\hat{T}^h(s)}(\mathrm{d}\phi)\mathrm{d}\hat{T}^h(s)
                \to\int^t_0\int_{\mathcal{U}}\sigma(\hat{\xi}^\delta(s),\hat{\alpha}^\delta(s),\phi)\hat{M}_{\hat{T}(s)}(\mathrm{d}\phi)\mathrm{T}(s).
            \end{aligned}
          \end{equation}
          Combining \eqref{varphim}--\eqref{sigmadelta}, we have 
          \begin{equation}
            \begin{aligned}
                \hat{x}(t)&=x+\int_0^t\int_{\mathcal{U}}b(\hat{x}(s),\hat{\alpha}(s),\phi)\hat{m}_{\hat{T}(s)}(\mathrm{d}\phi)\mathrm{d}\hat{T}(s)\\
               & +\int_0^t\int_{\mathcal{U}}\sigma(\hat{x}^\delta,\hat{\alpha}^\delta(s),\phi)\hat{M}_{\hat{T}(s)}(\mathrm{d}\phi)\mathrm{d}\hat{T}(s)+\varepsilon^\delta(t),
            \end{aligned}
          \end{equation}
          where $\lim_{\delta\to0}\mathbb{E}\vert\varepsilon^\delta(t)\vert=0$. Finally, taking 
          limits in the above equation as $\delta\to0$, \eqref{hatx} is obtained.
      \end{proof}
      
      \begin{theorem}\label{inverseJ}
          For $t<\infty$, the inverse can be defined as 
          \begin{align}
              \mathcal{T}(t)=\inf\{s:\hat{T}(s)>t\}.
          \end{align}
          Then $\mathcal{T}(t)$ is right continuous. $\mathcal{T}(t)\to\infty$ as $t\to\infty \mathrm{, w.p.1}$. 
          Define the rescaled process $\varphi(\cdot)$ by $\varphi(t)=\hat{\varphi}(\mathcal{T}(t))$ for any process 
          $\hat{\varphi}(\cdot)$. Then, $W(\cdot)$ is a standard $\mathcal{F}_t$-Wiener process and \eqref{surplus process} 
          holds.
      \end{theorem}

      \begin{proof}
        Since $\hat{T}(t)\to\infty$ w.p.1 as $t\to\infty$, $\mathcal{T}(t)$ exists for 
        all $t$ and $\mathcal{T}(t)\to\infty$ as $t\to\infty$ w.p.1. Similar to \eqref{exhto0} 
        and \eqref{exhtW}, for each $i\in\mathcal{M}$,
        \begin{equation}
            \begin{aligned}
                &\mathbb{E}K(\xi^h(t_k),\alpha^h(t_k),W^h(t_k),(\varphi_j,m^h)_{t_k},z^h(t_k),g^h(t_k),j\le\kappa,k\le p)
                \times [W(t+\delta)-W(t)]=0,\\
                &\mathbb{E}K(\xi^h(t_k),\alpha^h(t_k),W^h(t_k),(\varphi_j,m^h)_{t_k},z^h(t_k),g^h(t_k),j\le\kappa,k\le p)\\
                &\times [W^2(t)+\delta-W^2(t)-(\mathcal{T}(t+\delta)-\mathcal{T}(t))]=0.
            \end{aligned}
        \end{equation}
        Thus, we can verify that $W(\cdot)$ is an $\mathcal{F}_t$-Wiener process. A rescaling of
        \eqref{hatx} yields 
        \begin{equation}
            \begin{aligned}
                x(t) = x +\int^t_0\int_{\mathcal{U}}b(x(s),\alpha(s),\phi)m_s(\mathrm{d}\phi)\mathrm{d}s 
                +\int^t_0\int_{\mathcal{U}}\sigma(x(s),\alpha(s),\phi)M_s(\mathrm{d}\phi)\mathrm{d}s.
            \end{aligned}
        \end{equation}
        In other words, \eqref{surplus process} holds.

      \end{proof}
      
      \subsection{Convergence of the optimal value}
      To show the convergence of the optimal value of the objective function, we 
      find a comparison $\varepsilon$-optimal value next.

      \begin{lemma}\label{epsilonlemma}
          For each $\varepsilon>0$, there exists a continuous feedback control $u^\varepsilon(\cdot)$ 
          that is $\varepsilon$-optimal to all admissible controls. The solution to \eqref{surplus process} 
          is unique in a weak sense and has a unique invariant measure under this $\varepsilon$-optimal 
          control.
      \end{lemma}

      \begin{proof}
          The existence of a smooth $\varepsilon$-optimal can be guaranteed by modifying the method 
          in \cite{kushner1978optimality} for our formulation.
      \end{proof}

      \begin{theorem}
          Assume the conditions of Theorems \ref{limitx} and \ref{inverseJ} are satisfied. 
          Then as $h\to 0$,
          \begin{align}\label{gammato}
              \bar{\gamma}^h(x,i)\to\bar{\gamma}.
          \end{align}
      \end{theorem}

      \begin{proof}
          First, to prove 
          \begin{align}
              \bar{\gamma}^h(x,i)\le\bar{\gamma}.
          \end{align}
          Let $\tilde{u}(\cdot)$ be the optimal control and $\tilde{m}^h(\cdot)$ be the relaxed 
          control representation of $\tilde{u}^h(\cdot)$. Then $\bar{\gamma}^h(x,i)=\gamma^h(\tilde{u}^h)$. 
          Hence, in view of \eqref{gammaexp},
          \begin{equation}\label{gammale}
              \begin{aligned}
                  \bar{\gamma}^h(x,i)&=\mathbb{E}^{\tilde{u}^h}\int^1_0 f(\xi^h(s),\alpha^h(s),u(\xi^h(s),\alpha^h(s)))\mathrm{d}s\\
                  &=\mathbb{E}^{\tilde{u}^h}\int^1_0\int_{\mathcal{U}}f(\xi^h(s),\alpha^h(s),\phi)\tilde{m}^h_s(\mathrm{d}\phi)\mathrm{d}s\\
                  &\to\mathbb{E}^{\tilde{m}}\int^1_0\int_{\mathcal{U}}f(x(s),\alpha(s),\phi)\tilde{m}_s(\mathrm{d}\phi)\mathrm{d}s\\
                  &=\lim_{T}\frac{1}{T}\mathbb{E}^{\tilde{m}}\int^T_0\int_{\mathcal{U}}f(x(s),\alpha(s),\phi)\tilde{m}_s(\mathrm{d}\phi)\mathrm{d}s\\
                  &=\gamma(\tilde{m})\\
                  &\le\bar{\gamma},
              \end{aligned}
          \end{equation}
          where $\gamma(\tilde{m})$ is the optimal value of the performance function 
          for the limit stationary process.

          On the other hand, from Lemma \ref{epsilonlemma}, we have $\varepsilon$-optimal control 
          $u^\varepsilon$ such that
          \begin{equation}\label{gammage}
              \begin{aligned}
                  \bar{\gamma}^h(x,i)&\ge\gamma^h(u^\varepsilon,i)\\
                  &=\mathbb{E}^{u^\varepsilon}\int^1_0f(\xi^h(s),\alpha^h(s),u^\varepsilon(\xi^h(s),\alpha^h(s)))\mathrm{d}s\\
                  &\to\mathbb{E}^{u^\varepsilon}\int^1_0f(x(s),\alpha(s),u^\varepsilon(x(s),\alpha(s)))\mathrm{d}s\\
                  &=\lim_T\frac{1}{T}\mathbb{E}^{u^\varepsilon}\int^T_0f(x(s),\alpha(s),u^\varepsilon(x(s),\alpha(s)))\mathrm{d}s\\
                  &=\gamma(u^\varepsilon,i)\\
                  &\ge\bar{\gamma}-\varepsilon.
              \end{aligned}
          \end{equation}
          Combining \eqref{gammale} and \eqref{gammage} yields \eqref{gammato}.
      \end{proof}
    
    \section{Numerical algorithm}
        In this part, we will adopt and appropriately modify the hybrid deep learning Markov 
        chain approximation algorithm proposed by \cite{ChengJinYang2020}. This method 
        combines the numerical stability of MCAM, the functional generalization and 
        computational efficiency of neural networks, thereby avoiding the computational 
        explosion inherent in traditional approaches. We use a neural network to model the relationship between 
        the control strategy $u=(a,s,l)$ and the state value $(x,i)$. Neural networks possess 
        excellent nonlinear function fitting capabilities, enabling them to fit the optimal 
        control corresponding to discrete states and learn to output the optimal control 
        corresponding to continuous states. The core process of this method involves first 
        performing dynamic programming on a coarse state lattice with low precision to 
        ensure appropriate initial values, followed by a neural network optimization algorithm 
        on a fine state lattice with high precision to output a continuous control strategy.
        
        \subsection{Dynamic programming on a coarse state lattice}
        When solving the insurer's optimal reinsurance-investment-dividend control problem, 
        a core challenge lies in the high dimensionality of the state and control spaces. 
        Once both are discretized simultaneously, the computational load expands rapidly. 
        To alleviate the curse of dimensionality, neural networks are often used for the 
        approximation and optimization of continuous control. However, relying solely 
        on gradient-based training tends to get trapped in local optimization and 
        cause training instability. Even with improved gradient descent algorithms such as 
        the Adam optimizer, a large number of iterations are usually still required to 
        escape local optimization.

        The hybrid deep learning Markov chain approximation algorithm is designed precisely 
        to address this issue. It first uses dynamic programming, which has stable convergence, 
        to obtain a roughly shaped optimal strategy on coarsely discretized state lattices.
        Then, it takes this coarse solution as the initialization signal for the neural network 
        and performs refined optimization of continuous control on a finely discretized state 
        lattices. This two-stage mechanism not only avoids the instability caused by random 
        initialization, but also significantly enhances computational efficiency and convergence 
        robustness.

        We adopt two discretization parameters---coarse 
        and fine---to discretize the state space. The fine discretization state space is denoted as 
        $\{x_i\}^n_{i=1}$, and the coarse discretization state space is denoted as $\{y_j\}^m_{j=1}$, 
        such that $m<n$ and the discretization parameters $\Delta x<\Delta y$. Without
        loss of generality, the coarse state space can be taken as a subset of the fine 
        state space. 
        Because our objective is to maximize the long-run average reward, we employ 
        the Relative Value Iteration (RVI) rather than the standard value iteration on the 
        coarse lattice $\{y_j\}^m_{j=1}$ and obtain a rough solution.
        Denote the value function of $\{y_j\}^m_{j=1}$ as $U(y)$.  
        The relative value iteration method, introduced in \cite{JinYangYin2018} is described 
        as follows.
        \begin{enumerate}
            \item Set $t=0$. $\forall(y_j,\alpha_l)$, initialize value function and relative value function  
            $U_0(y_j,\alpha_l)=\tilde{U}_0(y_j,\alpha_l)=0$.
            \item \label{step2}Select a $y_0$, traverse $\{y_j\}^m_{j=1}$, and search for the optimal strategy 
            for each grid as follows:
            \begin{equation}\label{rvi1}
                \begin{aligned}
                    \tilde{u}^\ast_{t+1}&=\mathop{\arg\max}\limits_u S(y_j,\alpha_l,U_{t},u),\\
                    S(y_j,\alpha_l,U_{t},u)&=p((y_j,\alpha_l),(y_{j+1},\alpha_l)\vert u)\tilde{U}_{t}(y_{j+1},\alpha_l)
                    +p((y_j,\alpha_l),(y_{j-1},\alpha_l)\vert u)\tilde{U}_{t}(y_{j-1},\alpha_l)\\
                    &+\sum_{s\neq l}p((y_j,\alpha_l),(y_j,\alpha_s)\vert u)\tilde{U}_{t}(y_j,\alpha_s)
                    +f(y_j,\alpha_l,u)\Delta t(y_j,\alpha_l,u)
                \end{aligned}
            \end{equation}
            where 
            \begin{align}\label{rvi2}
                \tilde{U}_{t}(y_j,\alpha_l)=U_{t}(y_j,\alpha_l)-U_{t}(y_0,\alpha_l),
            \end{align}
            transition probabilities $p(\cdot\vert u)$ and interpolation of 
            time $\Delta t$ are shown in \eqref{probabilities}.
            \item Update $U_{t+1}(y_j,\alpha_l)=S(y_j,\alpha_l,U_{t},u^\ast_t)$. If $\max\vert U_{t+1}(y_j,\alpha_l)
            -U_{t}(y_j,\alpha_l)\vert<\varepsilon_r$, the iteration ends and the optimal strategy 
            $\tilde{u}=u^\ast_{t+1}$ and value function $U_{t+1}$ are output. 
            Otherwise, update $t\to t+1$, return to step \ref{step2}.
        \end{enumerate}

        \begin{remark}
            A well-specified admissible strategy set substantially reduces futile search 
            and improves numerical stability and convergence speed. Since the admissible 
            strategy set has been defined in \eqref{admissible strategy set}, we enforce 
            the same constraints when running RVI on the coarse lattice: 
            $a\in[M_a,1]$, $s\in[0,M_s]$, $l\in[M_l,1]$, $s+l\le 1$, and $s=l=0$ whenever $x\le K$. 
            Embedding feasibility before Bellman evaluations prevents infeasible candidates 
            from entering the comparison, thereby stabilizing and accelerating the iteration.
        \end{remark}

        \begin{remark}
            To accurately characterize the behavior of the optimal strategy near $x=K$, 
            we align the step sizes of both the coarse and fine grids with K when setting 
            them, thereby making the regulatory barrier $K$ a common node of the two sets 
            of grids. This alignment operation significantly reduces numerical oscillations 
            and strategy chattering at the regulatory threshold, and improves the convergence 
            quality of RVI in ke regions.
            
        \end{remark}

        \begin{remark}
            The optimal long-run average reward solved by RVI is a global constant independent 
            of the initial state. Here, the reference state $y_0$ is only used for centralization 
            and does not affect the optimal strategy or convergence results. For ease of 
            implementation and reproducibility, we will fix an internal grid point as $y_0$ 
            in the numerical experiments of the next section.
        \end{remark}

        \subsection{Neural network optimization on a fine state lattice}
        After obtaining the coarse strategy $\tilde{u}$ by running RVI on the 
        coarse grid, we perform refined optimization of the strategy on the fine grid. 
        First, we need to take $\tilde{u}$ as the target and fit it using a neural network, 
        such that the neural network output sufficiently approximates $\tilde{u}$, thereby 
        obtaining the initial parameters of the neural network. The initial 
        parameters $\theta_0$ is as follows:
        \begin{align}
            \theta_0=\mathop{\arg\min}\limits_{\theta}\sum^m_{j=0}\sum^{m_0-1}_{l=0}[\tilde{u}(y_j,\alpha_l)-N(y_j,\alpha_l\vert \theta)]^2,
        \end{align}
        where $N(y_j\vert \theta)$ denotes the output value of the neural network, and 
        $\theta$ is the parameter of the neural network. We can obtain good initial 
        parameters by minimizing the mean squared error using the gradient descent 
        method. At this point, the neural network will maintain an optimal strategy output 
        almost identical to that of RVI.

        After obtaining the initial parameters of the neural network, since both the 
        inputs and outputs of the neural network are continuous, we no longer perform 
        pointwise optimization as in RVI. Instead, we focus on improving the strategy 
        across most states and conduct global optimization over the entire state space. 
        We define the global optimization objective function $G$ as 
        \begin{equation}
            \begin{aligned}
                G \triangleq G(V(x_1,\alpha_1),\cdots,V(x_n,\alpha_n))= \frac{1}{n}\sum^n_{i=1}V(x_i,\alpha_i).
            \end{aligned}
        \end{equation}
        We achieve global optimization by maximizing $G$ as
        \begin{equation}
            \begin{aligned}
                \theta^k= \arg\max\limits_{\theta}G,
            \end{aligned}
        \end{equation}
        where $k$ denote the $k$-th iteration. In each iteration, according to the chain rule, 
        the parameter gradients can be directly obtained using the backpropagation algorithm. 
        We can perform parameter optimization using the formula as
        \begin{equation}
            \begin{aligned}
                \theta^k_{z} = \theta^k_{z-1} + h_1 \cdot \left[\frac{\partial G(\cdot,\theta)}{\partial \theta}\vert_{\theta=\theta^k_{z-1}}\right],
            \end{aligned}
        \end{equation}
        where $h_1$ is the learning rate. The gradient descent algorithm will terminate 
        when 
        \begin{equation}
            \begin{aligned}
                \vert G(\cdot,\theta^k_{z})-G(\cdot,\theta^k_{z+1})\vert < \epsilon_4,
            \end{aligned}
        \end{equation}
        where $\epsilon_4$ is a small positive number.

        \begin{remark}
            Since the admissible ranges of our three control strategies are inconsistent when 
            designing the neural network, we have adopted a multi-layer perceptron 
            structure consisting of a shared feature layer plus three independent policy networks. Additionally, we incorporated separate 
            output functions in the forward propagation to prevent infeasible actions from 
            participating in training. For ease of understanding, we provide a simplified 
            neural network structure as shown in Figure \ref{figure: network}.
        \end{remark}

        \begin{figure}[h!]
            \centering
            \captionsetup{font={small, stretch=1.312}}
            \includegraphics[width=0.6\linewidth]{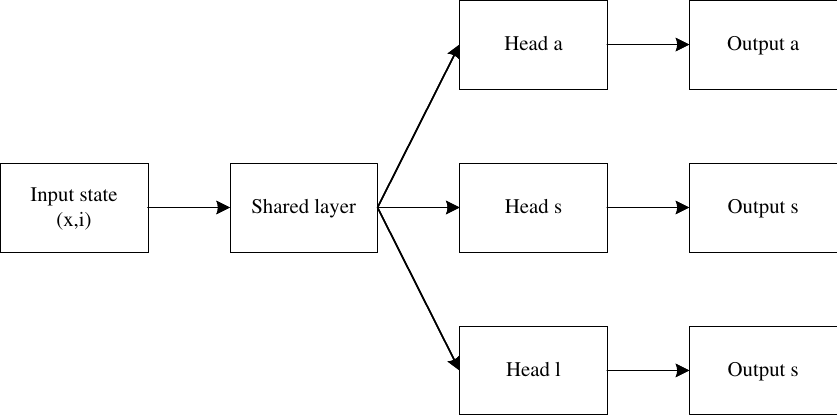}
            \caption{The neural network structure.}
            \label{figure: network}
        \end{figure}

        \subsection{Algorithmic procedure}
        \label{sec1}
        After the relative value iteration is run on the roughly divided state space to obtain 
        the rough control $\tilde{u}$, we consider using a neural network to fit $\tilde{u}$ on 
        the refined divided state space and further optimize it. The following is the global 
        algorithm. 
        \begin{enumerate}
            \item Select the boundary parameter $B$ of the surplus state space and the discretization parameters 
            $\Delta x$ and $\Delta y$, generate two equal difference sequences $\{x_i\}^n_{i=1}$ 
            and $\{y_j\}^m_{j=1}$, where $\{y_j\}^m_{j=1}$ is a subsequence of $\{x_i\}^n_{i=1}$ 
            such that $\Delta y=k\Delta x$ for $k$ is a positive integer, 
            and construct the regime-switching space $\{\alpha_l\}^{m_0-1}_{l=0}$ as \eqref{regime-switching}.
            \item Let the convergence accuracy and maximum iteration times of global iteration be 
            denoted as $\varepsilon_1$ and $w_1$, respectively. 
            Then the convergence accuracy of relative value iteration, neural network initial 
            parameter fitting, and neural network optimization are $\varepsilon_2$, $\varepsilon_3$, and 
            $\varepsilon_4$, respectively.
            \item Initialize the value function $V^0(x_i,\alpha_l)=U^0(y_j,\alpha_l)=0$ for all 
            $x_i\in\{x_i\}^n_{i=1}$,$y_j\in\{y_j\}^m_{j=1}$ and $\alpha_l\in\{\alpha_l\}^{m_0-1}_{l=0}$.
            \item In the $k$th round of global iteration, 
            \begin{enumerate}
                \item \label{stepa}Run the relative value iteration algorithm on $\{y_j\}^m_{j=1}$ as 
                \eqref{rvi1} and \eqref{rvi2}, then the 
                value function $U^{k-1}$ on the coarse grid is obtained from the previous global iteration. When 
                $U^k_t(y_j,\alpha_l)$ satisfies 
                \begin{align}
                    \max_{j,l}\vert U^k_t(y_j,\alpha_l)-U^k_{t-1}(y_j,\alpha_l)\vert<\varepsilon_2, 
                \end{align}
                the relative value iteration algorithm is terminated, and the rough control 
                $\tilde{u}^k(y_j,\alpha_l)$ in the rough partition state space is output. 
                \item Obtain initial parameters by neural network fitting as follows: 
                \begin{align}
                    \theta^k_0=\mathop{\arg\min}\limits_{\theta}\sum^m_{j=1}\sum^{m_0-1}_{l=0}[\tilde{u}^k(y_j,\alpha_l)-N(y_j,\alpha_l\vert \theta)]^2,
                \end{align}
                where $N(y_j,\alpha_l\vert\theta)$ is the output of neural network with parameters $\theta$ and inputs $y_j$ and $\alpha_l$.
                When the change value of mean square error loss is less than $\varepsilon_3$, the 
                fitting iteration is terminated.
                \item Use neural network for global optimization on $\{x_i\}^n_{i=1}$, and the 
                gradient descent algorithm to maximize the global optimization objective as
                \begin{equation}
                \begin{aligned}
                   G=\frac{1}{n}\sum^n_{i=1}\sum^{m_0-1}_{l=0}S(x_i,\alpha_l,V^{k-1},N(y_j,\alpha_l\vert \theta)),
                   \theta^k=\mathop{\arg\max}\limits_{\theta}G,
                \end{aligned}
                \end{equation}
                where $V^{k-1}$ is obtained from the previous global iteration, 
                the form $S(x_i,\alpha_l,V^{k-1},N(y_j,\alpha_l\vert \theta^k))$ is similar to \eqref{rvi1}. When the different value 
                of the optimization objective $\Delta G<\varepsilon_4$, the optimization iteration 
                is terminated, and the optimal strategy $\hat{u}(x_i,\alpha_l)=N(x_i,\alpha_l\vert \theta^k)$ is output.
                \item Update the value functions 
                \begin{equation}
                    \begin{aligned}
                        V^k(x_i,\alpha_l)=S(x_i,\alpha_l,V^{k-1},\hat{u}(x_i,\alpha_l)),
                        U^k(y_j,\alpha_l)=S(y_j,\alpha_l,U^{k-1},\hat{u}(y_j,\alpha_l)).
                    \end{aligned}
                \end{equation}
                \item When $\sum^n_{i=1}\sum^{m_0}_{l=1}\vert V^k(x_i,\alpha_l)-V^{k-1}(x_i,\alpha_l)\vert<\varepsilon_1$ 
                or the maximum number $w_1$ of global iterations is reached, the cycle is terminated. 
                Otherwise, turn to step \ref{stepa} to continue the $(k+1)$th round.
            \end{enumerate}
        \end{enumerate}
        \begin{remark}
            Note that we need the state value function outside the state space to participate in the 
            operation when calculating $S(x,V,u)$ in a finite space. We can approximate the boundary 
            state as follows.
            \begin{align*}
                V(x_0-\Delta x)=2V(x_0)-V(x_0+\Delta x),\ \ V(x_n+\Delta x)=2V(x_n)-V(x_n-\Delta x).
            \end{align*}
        \end{remark}

        \begin{remark}
            Due to the properties of the long-run average reward function and the characteristics of 
            RVI, the initial value of the value function does not affect the final convergence result. 
            Therefore, we can set the initial value of the functions $U$ and $V$ to zero. In our numerical 
            experiments, we need to maintain both the value function $U$ and $V$ and the centered 
            value function $\tilde{U}$ and $\tilde{V}$. To ensure sufficient accuracy, the coarse 
            grid parameter $m$ should be set as large as possible, which refines the solution and 
            reduces discretization errors. However, a large $m$ leads to an extremely large search 
            domain for each iteration of the relative value iteration. Therefore, parallel computation 
            is crucial to significantly accelerate the computation speed on the coarse grid, thus 
            improving the overall efficiency of the experiment.
        \end{remark}

\section{Numerical example and further remarks} 
      Now, we will provide a numerical example to illustrate the feasibility of the proposed numerical algorithm,
      and further discussions will be conducted in this section. 

      \subsection{Numerical example}
      We present a numerical example in this section. The regime
      switching process is set to have two states for simplicity, with $\mathcal{M}=\{0,1\}$. 
      And the parameters are specified in Table \ref{parameters}.
        \begin{table}[htbp]
            \centering
            \tbl{Numerical solution parameters.} 
            {
                \begin{tabular}{c c c c c c c c}  
                    \toprule
                    Parameter & Value & Parameter & Value & Parameter   & Value  & Parameter & Value  \\
                    \midrule
                  $q_{01}$    & 0.05          & $\beta$  & 0.25  & $r_1$    & $[0.08, 0.05]$  &  $B$      & 10  \\
                    $q_{10}$  & 0.1       & $E(Y)$, $E(Y^2)$   & 1 &  $K$  & 2   &  $\Delta_y$ & 0.5  \\
                    $\lambda$  & $[0.13, 0.28]$ &  $\rho$  & 0.15  &$M_a$     & 0.4             & $\Delta_x$ & 0.1  \\
                   $\sigma_S$ & $[0.2, 0.4]$          &$r_2$      & 0.02    & $M_s$    & 0.3           & $M_l$      & 0.062 \\
                   \bottomrule
                \end{tabular}
            }
            \label{parameters} 
        \end{table}

      Through the numerical procedure outlined in Section \ref{sec1}, 
      the optimal value is obtained.  
      That is $\bar{\gamma}^h=0.28$. And the optimal strategies' images are shown in Figure \ref{figure: optimalcontrol_f_n}. 
      We also provide the value function and strategy convergence images, as in Figure \ref{figure: value function convergence} and Figure \ref{figure: optimalcontrol_convergence}, respectively.

      When $x\le K$, the model restricts the risk investment ratio $s=0$ and the dividend 
      ratio $l=0$. In regime $\alpha=0$, the optimal strategies appear more aggressive. The 
      retention ratio $a(x,0)$ remains largely stable with only minor fluctuations, indicating 
      that under lower claim intensity and lower volatility, the insurer does not need to rely 
      heavily on reinsurance. The risky investment ratio $s(x,0)$ becomes positive once the surplus 
      exceeds the regulatory threshold, then decreases somewhat and gradually levels off. Its 
      overall level is consistently higher than in regime $\alpha=1$, reflecting that in a market 
      with higher returns and lower risk, the insurer is willing to hold a larger share of risky 
      assets. The dividend ratio $l(x,0)$ turns positive above the regulatory threshold, and then declines slowly 
      and stabilizes, showing a preference to retain surplus for investment rather than distribute 
      large dividends when market conditions are favorable. 

      In regime $\alpha=1$, the optimal strategies are clearly more conservative. The retention 
      ratio $a(x,1)$ is higher than in regime $\alpha=0$ and shows some variation with surplus: 
      it stays high at low surplus levels, then decreases with rising surplus before reaching a 
      stable platform. This pattern reflects the fact that with higher claim intensity, 
      reinsurance is more expensive, making it optimal to retain more risk. The risky investment 
      ratio $s(x,1)$ also becomes positive when the surplus exceeds the regulation  barrier, but declines steadily with surplus 
      and then stabilizes at a level significantly lower than in regime $\alpha=0$, indicating 
      that in a market with lower returns, higher volatility, and higher claim intensity, the 
      insurer continues to reduce exposure to risky assets. The dividend ratio $l(x,1)$ turns 
      positive more quickly after the regulation barrier, then declines slowly and stabilizes, showing 
      almost the same pattern as in regime $\alpha=0$.
    
    \begin{figure}[h!]
        \centering
        
        \subfloat[Optimal reinsurance ratio $a$.]{
            \resizebox{0.45\linewidth}{!}{\includegraphics{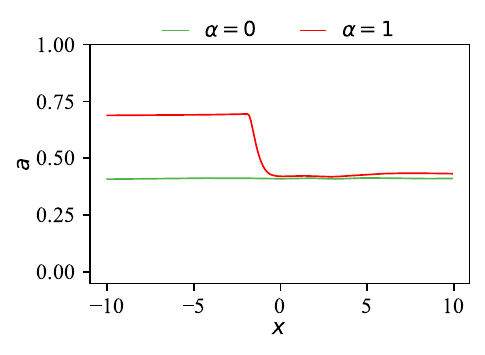}}
            \label{figure: a_f_n}
        }\hfill
        \subfloat[Optimal risky investment ratio $s$.]{%
            \resizebox{0.45\linewidth}{!}{\includegraphics{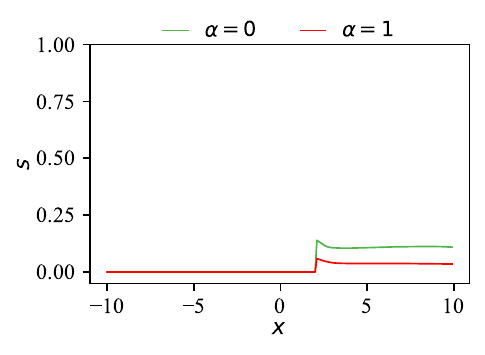}}
            \label{figure: s_f_n}
        }

        \vspace{0.3cm}
        \subfloat[Optimal dividend ratio $l$.]{%
            \resizebox{0.45\linewidth}{!}{\includegraphics{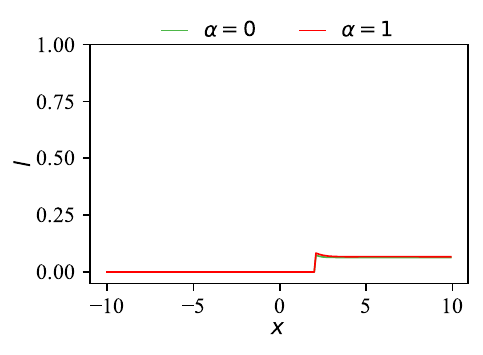}}
            \label{figure: l_f_n}
        }
        
        \caption[Optimal strategies with respect to the current surplus.]{Optimal strategies with respect to the current surplus.}
        \label{figure: optimalcontrol_f_n}
    \end{figure}

    \begin{figure}[h!]
        \centering
        \captionsetup{font={small, stretch=1.312}}
        \includegraphics[width=0.6\linewidth]{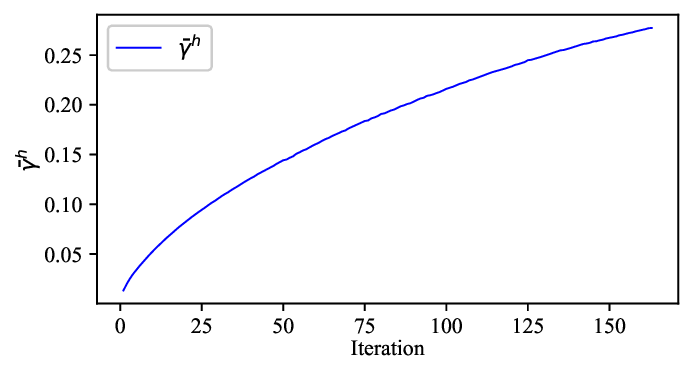}
        \caption{Optimal value function with respect to the iteration }
        \label{figure: value function convergence}
    \end{figure}

    \begin{figure}[h!]
        \centering
        
        \subfloat[Optimal reinsurance ratio $a$ under $\alpha=0$.]{
            \resizebox{0.45\textwidth}{!}{\includegraphics{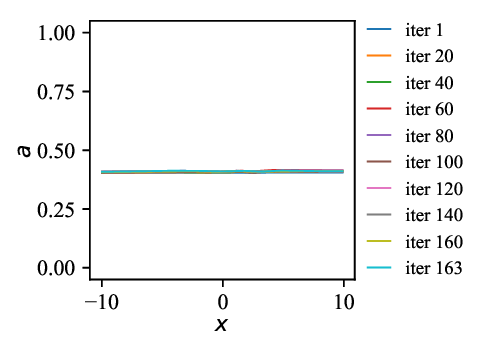}}
            \label{figure: a_c_alpha0}
        }\hfill
        \subfloat[Optimal risky investment ratio $a$ under $\alpha=1$.]{%
            \resizebox{0.45\textwidth}{!}{\includegraphics{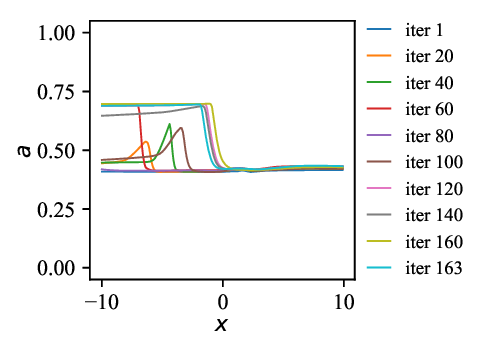}}
            \label{figure: a_c_alpha1}
        }

        \vspace{0.3cm}
        \subfloat[Optimal reinsurance ratio $s$ under $\alpha=0$.]{
            \resizebox{0.45\linewidth}{!}{\includegraphics{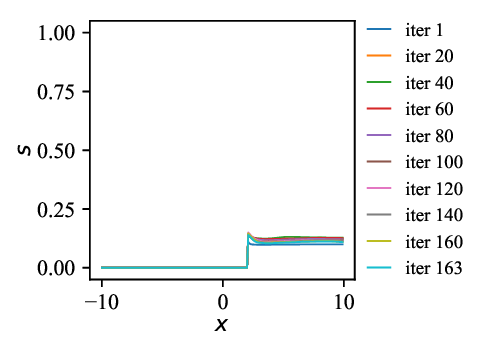}}
            \label{figure: s_c_alpha0}
        }\hfill
        \subfloat[Optimal risky investment ratio $s$ under $\alpha=1$.]{%
            \resizebox{0.45\linewidth}{!}{\includegraphics{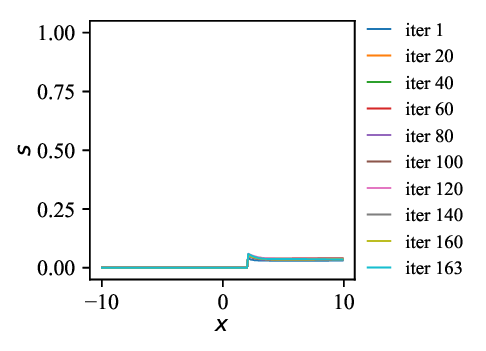}}
            \label{figure: s_c_alpha1}
        }

        \vspace{0.3cm}
        \subfloat[Optimal reinsurance ratio $l$ under $\alpha=0$.]{
            \resizebox{0.45\linewidth}{!}{\includegraphics{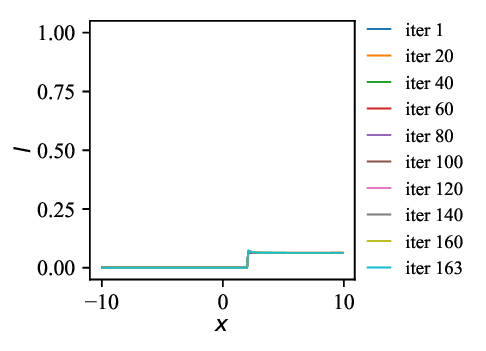}}
            \label{figure: l_c_alpha0}
        }\hfill
        \subfloat[Optimal risky investment ratio $l$ under $\alpha=1$.]{
            \resizebox{0.45\linewidth}{!}{\includegraphics{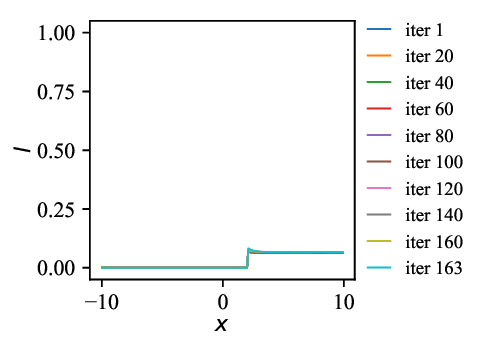}}
            \label{figure: l_c_alpha1}
        }
        \caption[Optimal strategies with respect to the iteration.]{Optimal strategies with respect to the iteration.}
        \label{figure: optimalcontrol_convergence}
    \end{figure}

      \subsection{Further remarks}
      This study is dedicated to establishing an integrated optimal control framework 
      for reinsurance, investment, and dividend strategies under solvency constraint (a given minimal cash requirement level) in a Markov regime-switching environment, with the core objective of maximizing 
      long-run average reward. The surplus process is governed by a continuous-time 
      Markov chain modeling shifts in the external environment, and a threshold-based 
      regulatory mechanism is introduced, permitting investment and dividend payments 
      only when the surplus level exceeds the pre-defined regulatory threshold.

      Given the analytical intractability of the coupled HJB equations arising 
      from the high-dimensional state space and regime-switching dynamics, a Markov 
      chain approximation method is adopted to discretize the original stochastic control 
      problem into a finite Markov decision process. Subsequently, a neural network-based 
      approximate dynamic programming approach is employed to obtain numerical solutions 
      in a computationally efficient manner. The convergence of the approximation scheme 
      is rigorously proven.

      For future research, the Markov chain approximation technique and the 
      neural policy iteration method developed here can be extended to a broader 
      class of constrained stochastic control problems featuring regulatory constraints, 
      regime-switching, and diffusion dynamics. Furthermore, efforts may be devoted to 
      optimal control problems with a long-run average objective subject to more 
      sophisticated and complex regulatory constraints.

\section*{Disclosure statement}
The authors report that there are no
competing interests to declare and there is no data set associated with the paper.

\section*{Funding}
The research was supported by the the National Natural Science Foundation of China (Grant No. 12001068). 

\bibliographystyle{apacite}
\bibliography{paper_zeng_li}

\end{document}